\documentclass[12pt,a4paper]{article}
\usepackage{amsfonts}
\usepackage{amsfonts}
\usepackage{amsfonts}
\usepackage{amsfonts}
\usepackage{amsfonts}
\usepackage{mathrsfs}
\usepackage{amsfonts}
\usepackage{mathrsfs}
\usepackage{amsfonts}
\usepackage{amsfonts}
\usepackage{mathrsfs}
\usepackage{mathrsfs}
\usepackage{amsfonts}
\usepackage{amsfonts}
\usepackage{amsfonts}
\usepackage{amsfonts}
\usepackage{mathrsfs}

\usepackage{amsfonts}
\usepackage{mathrsfs}
\usepackage{mathrsfs}
\usepackage{mathrsfs}
\usepackage{mathrsfs}
\usepackage{amsfonts}
\usepackage{mathrsfs}
\usepackage{mathrsfs}
\usepackage{amsfonts}
\usepackage{amsfonts}
\usepackage{amsfonts}
\usepackage{amsfonts}
\usepackage{amsfonts}
\usepackage{amsfonts}
\usepackage{amsfonts}
\usepackage{amsfonts}
\usepackage{amsfonts}
\usepackage{mathrsfs}
\usepackage{amsfonts}
\usepackage{mathrsfs}
\usepackage{amsfonts}
\usepackage{mathrsfs}
\usepackage{amsfonts}
\usepackage{amsfonts}
\usepackage{amsfonts}
\usepackage{amsfonts}
\usepackage{amsfonts}
\usepackage{amsfonts}
\usepackage{amsfonts}
\usepackage{amsfonts}
\usepackage{amsfonts}
\usepackage{amsfonts}
\usepackage{amsfonts}
\usepackage{amsfonts}
\usepackage{mathrsfs}
\usepackage{mathrsfs}
\usepackage{mathrsfs}
\usepackage{mathrsfs}
\usepackage{mathrsfs}
\usepackage{mathrsfs}
\usepackage{mathrsfs}
\usepackage{mathrsfs}
\usepackage{mathrsfs}
\usepackage{amsfonts}
\usepackage{amsfonts}
\usepackage{amsfonts}
\usepackage{amsfonts}
\usepackage{amsfonts}
\usepackage{amsfonts}
\usepackage{mathrsfs}
\usepackage{amsfonts}
\usepackage{amsfonts}
\usepackage{amsfonts}
\usepackage{amsfonts}
\usepackage{amsfonts}
\usepackage{amsfonts}
\usepackage{amsfonts}
\usepackage{amsfonts}
\usepackage{amsfonts}
\usepackage{amsfonts}
\usepackage{amsfonts}
\usepackage{amsfonts}
\usepackage{mathrsfs}
\usepackage{amsfonts}
\usepackage{mathrsfs}
\usepackage{amsfonts}
\usepackage{amsfonts}
\usepackage{amsfonts}
\usepackage{mathrsfs}
\usepackage{amsfonts}
\usepackage{amsfonts}
\usepackage{amsfonts}
\usepackage{amsfonts}
\usepackage{amsmath}
\usepackage{mathrsfs}
\usepackage{amsfonts}
\usepackage{amsfonts}
\usepackage{amsfonts}
\usepackage{amsfonts}
\usepackage{amsfonts}
\usepackage{amsfonts}
\usepackage{amsfonts}
\usepackage{amsfonts}
\usepackage{amsfonts}
\usepackage{amsfonts}
\usepackage{amsfonts}
\usepackage{mathrsfs}
\usepackage{amsfonts}
\usepackage{amsfonts}
\usepackage{amsfonts}
\usepackage{amsfonts}
\usepackage{amsmath}
\usepackage{mathrsfs}
\usepackage{mathrsfs}
\usepackage{mathrsfs}
\usepackage{mathrsfs}
\usepackage{mathrsfs}
\usepackage{mathrsfs}
\usepackage{mathrsfs}
\usepackage{mathrsfs}
\usepackage{amsfonts}
\usepackage{amsfonts}
\usepackage{amsfonts}
\usepackage{amsfonts}
\usepackage{amsfonts}
\usepackage{amsfonts}
\usepackage{amsfonts}
\usepackage{amsfonts}
\usepackage{amsfonts}
\usepackage{amsfonts}
\usepackage{amsfonts}
\usepackage{amsfonts}
\usepackage{mathrsfs}
\usepackage{amsfonts}
\usepackage{mathrsfs}
\usepackage{amsfonts}
\usepackage{amsfonts}
\usepackage{amsfonts}
\usepackage{amsfonts}
\usepackage{amsfonts}
\usepackage{amsfonts}
\usepackage{amsfonts}
\usepackage{amsfonts}
\usepackage{amsfonts}
\usepackage{amsfonts}
\usepackage{amsfonts}
\usepackage{amsfonts}
\usepackage{amsfonts}
\usepackage{amsfonts}
\usepackage{amsfonts}
\usepackage{amsfonts}
\usepackage{amsfonts}
\usepackage{amsfonts}
\usepackage{mathrsfs}
\usepackage{mathrsfs}
\usepackage{mathrsfs}
\usepackage{mathrsfs}
\usepackage{amsfonts}
\usepackage{amsfonts}
\usepackage{amsfonts}
\usepackage{amsfonts}
\usepackage{amsfonts}
\usepackage{amsfonts}
\usepackage{amsfonts}
\usepackage{amsfonts}
\usepackage{amsfonts}
\usepackage{amsfonts}
\usepackage{amsfonts}
\usepackage{amsfonts}
\usepackage{amsfonts}
\usepackage{amsfonts}
\usepackage{amsfonts}
\usepackage{amsfonts}
\usepackage{amsfonts}
\usepackage{mathrsfs}
\usepackage{mathrsfs}
\usepackage{amsfonts}
\usepackage{amsfonts}
\usepackage{amsfonts}
\usepackage{amsfonts}
\usepackage{amsfonts}
\usepackage{amsfonts}
\usepackage{amsfonts}
\usepackage{amsfonts}
\usepackage{amsfonts}
\usepackage{amsfonts}
\usepackage{amsfonts}
\usepackage{amsfonts}
\usepackage{amsfonts}
\usepackage{mathrsfs}
\usepackage{mathrsfs}
\usepackage{amsfonts}
\usepackage{amsfonts}
\usepackage{amsfonts}
\usepackage{amsfonts}
\usepackage{amsfonts}
\usepackage{amsfonts}
\usepackage{mathrsfs}
\usepackage{mathrsfs}
\usepackage{amsfonts}
\usepackage{amsfonts}
\usepackage{amsfonts}
\usepackage{amsfonts}
\usepackage{amsfonts}
\usepackage{amsfonts}
\usepackage{amsfonts}
\usepackage{amsfonts}
\usepackage{amsfonts}
\usepackage{amsfonts}
\usepackage{amsfonts}
\usepackage{amsfonts}
\usepackage{amsfonts}
\usepackage{amsfonts}
\usepackage{mathrsfs}
\usepackage{amsfonts}
\usepackage{amsfonts}
\usepackage{amsfonts}
\usepackage{amsfonts}
\usepackage{amsfonts}
\usepackage{amsfonts}
\usepackage{amsfonts}
\usepackage{amsfonts}
\usepackage{amsfonts}
\usepackage{amsfonts}
\usepackage{amsfonts}
\usepackage{amsfonts}
\usepackage{amsfonts}
\usepackage{amsfonts}
\usepackage{amsfonts}
\usepackage{amsfonts}
\usepackage{amsmath}
\usepackage{mathrsfs}
\usepackage{amsfonts}
\usepackage{amsfonts}
\usepackage{amsfonts}
\usepackage{amsfonts}
\usepackage{amsmath}
\usepackage{mathrsfs}
\usepackage{amsfonts}
\usepackage{mathrsfs}
\usepackage{mathrsfs}
\usepackage{mathrsfs}
\usepackage{mathrsfs}
\usepackage{amsmath}
\usepackage{mathrsfs}
\usepackage{amsfonts}
\usepackage{color}
\usepackage{mathrsfs}

\usepackage{stmaryrd}
\usepackage{amsmath}
\usepackage{amssymb}

\usepackage{bbm}
\usepackage{mathrsfs}
\usepackage{graphicx}

\usepackage{enumerate}
\usepackage{theorem}

\makeatletter
\def\tank#1{\protected@xdef\@thanks{\@thanks
        \protect\footnotetext[0]{#1}}}
\def\bigfoot{

    \@footnotetext}
\makeatother

\newcommand{\ea}{\end{array}}
\newtheorem{theorem}{Theorem}[section]

\newtheorem{claim}{Claim}[section]

\newtheorem{lemma}{Lemma}[section]
\newtheorem{definition}{Definition}[section]
\newtheorem{remark}{Remark}[section]

{\theorembodyfont{\rmfamily}
}

\newenvironment{proof}{Proof.}

\def \eref#1{\hbox{(\ref{#1})}}

\allowdisplaybreaks

\begin{document}
\title{\Large \bf Large Deviations for Stochastic Porous Media Equation on General
Measure Spaces\thanks{Weina Wu's research is supported by National Natural Science Foundation of China (NSFC) (No. 11901285, 11771187), China Scholarship Council (CSC) (No. 202008320239), School Start-up Fund of Nanjing University of Finance and Economics (NUFE), Support Programme for Young Scholars of NUFE.
Jianliang Zhai's research is supported by National Natural Science Foundation of China (NSFC) (No. 11971456, No. 11671372, No. 11721101), School Start-up Fund (USTC) KY0010000036.} }

\author{{Weina Wu}$^a$\footnote{E-mail:wuweinaforever@163.com}~~~ {Jianliang Zhai}$^b$\footnote{E-mail:zhaijl@ustc.edu.cn}
\\
 \small  a. School of Economics, Nanjing University of Finance and Economics, Nanjing, Jiangsu 210023, China.\\
 \small  b.  CAS Wu Wen-Tsun Key Laboratory of Mathematics, School of Mathematical Sciences,\\
 \small University of Science and Technology of China, Hefei, Anhui 230026, China.}\,
 %University of Science and Technology of China, Hefei, Anhui, 230026, PR China
\date{}
\maketitle

\begin{center}
\begin{minipage}{130mm}
{\bf Abstract.} In this paper, we establish a large deviation principle for stochastic porous media equations driven by time-dependent multiplicative noise on a $\sigma$-finite measure space $(E,\mathcal{B}(E),\mu)$, and with the Laplacian replaced by a negative definite self-adjoint operator. The coefficient $\Psi$ is only assumed to satisfy the increasing Lipschitz nonlinearity assumption without the restriction $r\Psi(r)\rightarrow\infty$ as $r\rightarrow\infty$ for $L^2(\mu)$-initial data. This paper also gets rid of the compact embedding assumption on the associated Gelfand triple. Examples of the negative definite self-adjoint operators include fractional powers of the Laplacian, i.e. $L=-(-\Delta)^\alpha,\ \alpha\in(0,1]$, generalized Schr\"{o}dinger operators, i.e. $L=\Delta+2\frac{\nabla \rho}{\rho}\cdot\nabla$, and Laplacians on fractals.

%Another merit of this paper is to get rid of the compact embedding assumption on the associated Gelfand triple.

\vspace{3mm} {\bf Keywords:} Porous media equation;
Sub-Markovian; Strongly continuous contraction semigroup; Weak convergence method; Large deviations.

\end{minipage}
\end{center}

\section{Introduction}
\setcounter{equation}{0}
 \setcounter{definition}{0}

Let $(E,\mathcal{B}(E),\mu)$ be a standard measurable space (\cite[page:133, Definition 2.2]{P67}) with a $\sigma$-finite measure. Let $W$ be an $L^2(\mu)$-valued cylindrical Wiener process defined on a given complete filtered probability space $(\Omega,\mathcal{F},\mathbb{F},\Bbb{P})$, where $\mathbb{F}:=\{\mathcal{F}_t\}_{t\geq0}$, satisfying the usual condition. The intention of this paper is to prove a large deviation principle (LDP) for the following stochastic generalized porous media equations with small noise:
%\begin{equation} \label{eq:1}
%dX^\varepsilon(t)-L\Psi(X^\varepsilon(t))dt= \sqrt{\varepsilon}B(t,X^\varepsilon(t))dW(t),\ \text{in}\ [0,T]\times E,\ \varepsilon>0,
%\end{equation}
\begin{equation} \label{eq:1}
\left\{ \begin{aligned}
&dX^\varepsilon(t)-L\Psi(X^\varepsilon(t))dt= \sqrt{\varepsilon}B(t,X^\varepsilon(t))dW(t),\ \text{in}\ [0,T]\times E,\ \varepsilon>0,\\
&X^\varepsilon(0)=x \text{~on~} E,
\end{aligned} \right.
\end{equation}
where $L$ is the infinitesimal generator of a symmetric sub-Markovian strongly continuous contraction semigroup $(P_t)_{t\geq0}$ on $L^2(\mu)$.
$\Psi(\cdot):\Bbb{R}\rightarrow\Bbb{R}$ is a monotonically nondecreasing Lipschitz continuous function, $B$ is measurable and takes values in the space of Hilbert-Schmidt operators $L_2(L^2(\mu),F^*_{1,2})$. For the definition of the Hilbert space
$F^*_{1,2}$ and the precise conditions on $\Psi$ and $B$ we refer to Section 2 and Section 3 respectively.

The study of the asymptotic behavior of stochastic porous media equations is one of the most important problems of modern mathematical physics. One way to investigate the problem is to consider its LDPs.
In \cite{L,RWW}, LDPs for \eref{eq:1} were studied in the following framework: Let $(E,\mathcal{B}(E),\mu)$ be a separable probability space, $L$ a negative definite self-adjoint linear operator on $L^2(\mu)$ such that $L^{-1}$ is bounded on $L^{r+1}(E,\mathcal{B}(E),\mu)$, for some $r>1$. In \cite{RWW}, the authors used Schilder's theorem for Gaussian processes and approximation procedures to establish LDPs for (\ref{eq:1}) with additive noise. By applying a weak convergence method, LDPs for stochastic partial differential equations with general monotone drift driven by multiplicative Gaussian noises and L\'evy noises were obtained in \cite{L,XZ} respectively. However, these results do not apply to our framework, since $(E,\mathcal{B}(E),\mu)$ is assumed to be a $\sigma$-finite measure space.
In this paper, the existence and uniqueness of solutions to \eref{eq:1} is based on the recent paper \cite{RWX}, more precisely, we consider \eref{eq:1} under the Gelfand triple $L^2(\mu)\subset F^*_{1,2}\subset (L^2(\mu))^*$, which avoids the assumption of compactness of embedding in the corresponding Gelfand triple, see \cite[pg:52]{L} and \cite[pg:2858]{XZ}. In addition, we keep the assumptions for $L$, $\Psi$ and $B$ as in \cite{RWX}. Hence, the examples given in \cite{RWX} also apply here, furthermore, our $L$ can cover all examples mentioned in \cite{RWX}, such as generalized Schr\"odinger operators, i.e., $L=\Delta+2\frac{\nabla \rho}{\rho}\cdot\nabla$, fractional powers of the Laplacian, i.e., $L=-(-\Delta)^\alpha$, $\alpha\in(0,1]$, and Laplacians on fractals. In particular, the result in \cite{RWX} generalizes that in \cite{BRR}, since they do not need the restriction on $d$ when $E=\mathbb{R}^d$ and $L=-(-\Delta)^{\alpha},\ \alpha\in(0,1]$. We would also like to mention that in \cite{RRW} and \cite{RW}, $\Psi$ is assumed to be continuous such that $r\Psi(r)\rightarrow\infty$ as $r\rightarrow\infty$, which is not needed for $L^2(\mu)$-initial data in \cite{RWX} and this paper.

To obtain a LDP for \eref{eq:1}, our method is based on a weak convergence approach introduced by \cite{BD}, which has been applied to various dynamical systems driven by Gaussian noises, see e.g. \cite{BD1, BDM1, BDM2, BGJ, CM, CR, L, RZ, RZZ}. Recently, a sufficient condition to verify the large deviation criteria of Budhiraja-Dupuis-Maroulas has been improved by Matoussi, Sabbagh and Zhang in \cite{MSZ}. This condition seems to be more suitable to deal with SPDEs arising from fluid mechanics; see e.g.\cite{DWZZ}. In this paper we will use this method. The main point of our procedures is to prove the convergence of some skeleton equations. Before this, we state the results on existence, uniqueness and provide some a priori estimates for the solutions to the skeleton equations. The corresponding proof is quite involved and we do not include the details to keep down its size.

Finally, we would like to refer \cite{LR, P, PR} for more background information and results on SPDEs, \cite{A, BDR} on SPMEs, \cite{RRW, RW, RWW, RWX} and the references therein for comprehensive theories of stochastic generalized porous media equations.

The paper is organised as follows. In Section 2, we introduce some notations and preliminaries. Hypothesis and main result will be given in Section 3. In Section 4, we prove the existence and uniqueness of solutions to the associated skeleton equations. The large deviation principle is proved in Section 5.

\smallskip

\section{Notations and Preliminaries}
\setcounter{equation}{0}
 \setcounter{definition}{0}

First of all, let us recall some basic definitions and spaces which will be used throughout the paper (see \cite{RWX}).

Let $(E,\mathcal{B}(E),\mu)$ be a $\sigma$-finite measure space, $(P_t)_{t\geq0}$ be a strongly continuous, symmetric sub-Markovian
contraction semigroup on $L^2(\mu):=L^2(E,\mathcal{B}(E),\mu)$ with generator $(L, D(L))$. The $\Gamma$-transform of $(P_t)_{t\geq0}$ is defined by the following Bochner integral (\cite{HK})
\begin{eqnarray}\label{eqnarray1}
V_ru:=\frac{1}{\Gamma(\frac{r}{2})}\int_0^\infty t^{\frac{r}{2}-1}e^{-t}P_tudt,~~u\in L^2(\mu),~~r>0.
\end{eqnarray}
In this paper, we consider the Hilbert space $(F_{1,2},\|\cdot\|_{F_{1,2}})$ defined by
$$
F_{1,2}:=V_1(L^2(\mu)),~\text{with~norm}~\|f\|_{F_{1,2}}=|u|_2,~~\text{for}~~f=V_1u,~~ u\in L^2(\mu),
$$
where the norm $|\cdot|_2$ is defined as $|u|_2=(\int_E |u|^2d\mu)^{\frac{1}{2}}$.
From \cite{FJS1,FJS2}, we know
$$
V_1=(1-L)^{-\frac{1}{2}},~~\text{so~that}~~F_{1,2}=D\big((1-L)^{\frac{1}{2}}\big)~~\text{and}~~\|f\|_{F_{1,2}}=|(1-L)^{\frac{1}{2}}f|_2.
$$
The dual space of $F_{1,2}$ is denoted by $F^*_{1,2}$ and $F^*_{1,2}=D((1-L)^{-\frac{1}{2}})$, it is equipped with norms
\begin{eqnarray}\label{eqnarray5}
\|\eta\|_{F^*_{1,2,\nu}}:=\langle \eta, (\nu-L)^{-1}\eta\rangle_2^{\frac{1}{2}},~~\eta\in F^*_{1,2},~~0<\nu<\infty.
\end{eqnarray}

Let $H$ be a separable Hilbert space with inner product $\langle\cdot,\cdot\rangle_H$ and $H^*$ its dual. Let $V$ be a reflexive Banach space such that $V\subset H$ continuously and densely. Then for its dual space $V^*$ it follows that $H^*\subset V^*$ continuously and densely. Identifying $H$ and $H^*$ via the Riesz isomorphism we have that
$$V\subset H\subset V^*$$
continuously and densely. If $_{V^*}\langle\cdot,\cdot\rangle_V$ denotes the dualization between $V^*$ and $V$ (i.e. $_{V^*}\langle z,v\rangle_V:=z(v)$ for $z\in V^*, v\in V$), it follows that
\begin{eqnarray}\label{eqn6}
_{V^*}\langle z,v\rangle_V=\langle z,v\rangle_H,~~\text{for~all}~z\in H,~~v\in V.
\end{eqnarray}
$(V,H,V^*)$ is called a Gelfand triple.

In \cite{RWX}, the authors constructed a Gelfand triple with $V=L^2(\mu)$ and $H=F^*_{1,2}$, the Riesz map which identifies $F_{1,2}$ and $F^*_{1,2}$ is $(1-L)^{-1}: F^*_{1,2}\rightarrow F_{1,2}$.

We need the following lemma which was proved in \cite{RWX}.
\begin{lemma}
The map
$$1-L:F_{1,2}\rightarrow F_{1,2}^*$$
extends to a linear isometry
$$1-L:L^2(\mu)\rightarrow(L^2(\mu))^*,$$
and for all $u,v\in L^2(\mu)$,
\begin{eqnarray}\label{eqn7}
_{(L^2(\mu))^*}\langle(1-L)u, v\rangle_{L^2(\mu)}=\int_Eu\cdot
v~d\mu.
\end{eqnarray}
\end{lemma}

\vskip 0.2cm
Now, some basic notations and definitions of large deviations need to be presented.

\vspace{2mm}

Let $\{\Gamma^\varepsilon\}_{\varepsilon>0}$ be a family of random variables defined on a given probability space $(\Omega, \mathcal{F},\Bbb{P})$ taking values in some Polish space $\mathcal{E}$. Let $\mathcal{B}(\mathcal{E})$ denotes the Borel $\sigma$-field of $\mathcal{E}$.
\vspace{2mm}
\begin{definition}
A function $I:\mathcal{E}\rightarrow[0,\infty]$ is called a rate function if $I$ is lower semi-continuous. A rate function $I$ is called a good rate function if the level set $\{e\in\mathcal{E}:I(e)\leq M\}$ is a compact subset of $\mathcal{E}$ for each $M<\infty$.
\end{definition}

\begin{definition}
Let $I$ be a rate function. The sequence $\{\Gamma^{\varepsilon}\}_{\varepsilon>0}$ is said to satisfy a LDP on $\mathcal{E}$ with rate function $I$ if for each Borel subset $\mathcal{O}$ of $\mathcal{E}$
$$-\inf_{e\in\mathcal{O}^0} I(e)\leq\lim\inf_{\varepsilon\rightarrow0}\varepsilon\log\Bbb{P}(\Gamma^{\varepsilon}\in \mathcal{O})\leq\lim\sup_{\varepsilon\rightarrow0}\varepsilon\log\Bbb{P}(\Gamma^{\varepsilon}\in \mathcal{O})\leq-\inf_{e\in\overline{\mathcal{O}}}I(e),$$
where $\mathcal{O}^0$ and $\overline{\mathcal{O}}$ denote the interior and closure of $\mathcal{O}$ in $\mathcal{E}$, respectively.
\end{definition}
\smallskip

Thoughout the paper, let $L^2([0,T]\times\Omega;L^2(\mu))$ denote
the space of all $L^2(\mu)$-valued square-integrable functions on
$[0,T]\times\Omega$, and $C([0,T];F^*_{1,2})$ the space of all
continuous $F^*_{1,2}$-valued functions on $[0,T]$. For two
Hilbert spaces $H_1$ and $H_2$, the space of Hilbert-Schmidt
operators from $H_1$ to $H_2$ is denoted by $L_2(H_1, H_2)$.
For any $J\in L_2(H_1, H_2)$ and $h_1\in H_1$, we use the symbol ``$J\circ h_1$" to
denote the value of the Hilbert-Schmidt
operator $J$ acting on $h_1$, and it is easy to see that $J\circ h_1\in H_2$.
 For
simplicity, the positive constants $c$, $C$, $C_i$, $i=1,2,3$ and $C_{T,M}$ used in this paper may change from line to line.

\section{Hypothesis and main results}
\setcounter{equation}{0}
 \setcounter{definition}{0}

In this paper, we study \eref{eq:1} with the following hypotheses:

\medskip
\noindent \textbf{(H1)} $\Psi(\cdot):\Bbb{R}\rightarrow \Bbb{R}$ is a monotonically nondecreasing Lipschitz function with $\Psi(0)=0$.

\vspace{1mm}
\medskip
\noindent \textbf{(H2)} $B(t,u):[0,T]\times L^2(\mu)\rightarrow L_2(L^2(\mu),L^2(\mu))$ satisfies
\vspace{2mm}

\noindent \textbf{(i)} there exists $C_1\in[0, \infty)$ such that
$$\|B(t, u)-B(t, v)\|_{L_2(L^2(\mu), F^*_{1,2})}\leq C_1\|u-v\|_{ F^*_{1,2}}\ \ \text{for\ all}\ t\in[0,T]\ \text{and}\ u, v\in L^2(\mu);$$

\noindent \textbf{(ii)} there exists $C_2\in(0, \infty)$ such that
$$\|B(t, u)\|_{L_2(L^2(\mu), F^*_{1,2})}\leq C_2(\|u\|_{ F^*_{1,2}}+1)\ \ \text{for\ all}\ t\in[0,T]\ \text{and}\ u\in L^2(\mu);$$

\noindent \textbf{(iii)} there exists $C_3\in(0, \infty)$ such that
$$\|B(t, u)\|_{L_2(L^2(\mu), L^2(\mu))}\leq C_3(|u|_2+1)\ \ \text{for\ all}\ t\in[0,T]\ \text{and}\ u\in L^2(\mu).$$

\vspace{1mm}
\medskip
\noindent \textbf{(H3)} There exists $\gamma>0$ such that
\begin{eqnarray*}\label{30}
\|B(t_1,u)-B(t_2,u)\|_{L_2(L^2(\mu),F^*_{1,2})}\leq C(\|u\|_{F^*_{1,2}}+1)|t_1-t_2|^{\gamma}~ \text{for~all}~u\in L^2(\mu),~ t_1, t_2\in [0,T].
\end{eqnarray*}

\vspace{2mm}
According to \cite[Theorem 3.1]{RWX}, we can instantly obtain the following theorem.
\begin{theorem}\label{Th 3.1}
Suppose that \textbf{(H1)} and \textbf{(H2)} hold. Then, for each $x\in L^2(\mu)$, there exists a unique strong solution $X^\varepsilon$ to \eref{eq:1}, i.e., $X^\varepsilon$ is an $F^*_{1,2}$-valued $\{\mathcal{F}_t\}_{t\geq0}$-adapted process, and the following conditions are satisfied:
\begin{eqnarray*}
X^\varepsilon\in L^2(\Omega,C([0,T];F_{1,2}^*))\cap L^2([0,T]\times\Omega;L^2(\mu)),
\end{eqnarray*}
\begin{eqnarray*}
\int_0^\cdot\Psi(X^\varepsilon(s))ds\in C([0,T];F_{1,2}),~~\Bbb{P}\text{-}a.s.,
\end{eqnarray*}
\begin{eqnarray*}
X^\varepsilon(t)-L\int_0^t\Psi(X^\varepsilon(s))ds=x+\int_0^t\sqrt{\varepsilon}B(t,X^\varepsilon(s))dW(s),~~\forall~t\in[0,T],~~\Bbb{P}\text{-}a.s..
\end{eqnarray*}
\end{theorem}

The intention of this paper is to prove a large deviation principle for (\ref{eq:1}), i.e., $X^\varepsilon$ on $C([0,T];F^*_{1,2})$.

Before the statement of the main result, we need to introduce the following skeleton equations for any $h\in L^2([0,T];L^2(\mu))$, which is used to define the good rate function:
\begin{eqnarray}\label{eq:29}
dY^{h}(t)-L\Psi(Y^{h}(t))dt=B(t,Y^{h}(t))\circ h(t)dt,
\end{eqnarray}
with initial value $x\in F_{1,2}^*$.

\begin{definition}
Given $h\in L^2([0,T];L^2(\mu))$, a function $Y^{h}$ is called a strong solution to
\eref{eq:29} if the following conditions are satisfied:
\begin{equation}\label{equ:1}
Y^{h}\in L^2([0,T];L^2(\mu))\cap
C([0,T];F_{1,2}^*),
\end{equation}
\begin{equation}\label{equ:2}
\int_0^\cdot \Psi(Y^{h}(s))ds\in C([0,T];F_{1,2}),
\end{equation}
and
\begin{equation}\label{equ:3}
Y^{h}(t)-L\int_0^t\Psi(Y^{h}(s))ds=x+\int_0^tB(s,Y^{h}(s))\circ h(s)ds,\ \forall\ t
\in[0,T]\ \ \text{in}\ \ F^*_{1,2}.
\end{equation}
\end{definition}
We state the following result, whose proof is provided in Section 4.
\begin{theorem}\label{Th 3.2}
Suppose that \textbf{(H1)} and \textbf{(H2)} are satisfied.
Then, for each $x\in L^2(\mu)$, there exists a unique strong solution $Y^{h}$ to
\eref{eq:29} and a positive constant $C_{T,M}\in(0,\infty)$ which depends on $T$ and $M$ satisfying
\begin{eqnarray*}
\sup_{h\in\mathcal{S}_M}\sup_{t\in[0, T]}\big|Y^{h}(t)\big|_2^2\leq C_{T,M}(|x|_2^2+1).
\end{eqnarray*}
Suppose that \textbf{(H1)}, \textbf{(H2)(i)}, \textbf{(H2)(ii)} and the following are satisfied,
\begin{equation*}
\Psi(r)r\geq c r^2,\ \forall\ r\in \mathbb{R},
\end{equation*}
where $c\in(0, \infty)$. Then, for all $x\in F_{1,2}^*$, there is a unique strong solution $Y^{h}$ to
\eref{eq:29}.
\end{theorem}

\vspace{2mm}

Our main theorem is as follows.

\begin{theorem}\label{Th 3.3}
Let $x\in L^2(\mu)$. Suppose \textbf{(H1)}-\textbf{(H3)} are satisfied. Then, the solution of \eref{eq:1}, i.e., $X^\varepsilon$, satisfies a LDP on $C([0,T];F^*_{1,2})$ with the following good rate function $I:C([0,T];F^*_{1,2})\rightarrow[0,\infty]$ defined by
\begin{eqnarray}\label{007}
I(f)=\inf%_{\{:f=\mathcal{G}^0(\int_0^\cdot h(s)ds)\}}
\big\{\frac{1}{2}\int_0^T|h(s)|_2^2ds:~f=Y^h,~h\in L^2([0,T];L^2(\mu))~\big\},
\end{eqnarray}
where $Y^h$ solves \eref{eq:29}. Here we use the convention that the infimum of an empty set is $\infty$.
\end{theorem}

\begin{proof}
Theorem \ref{Th 3.2} implies that there is a measurable mapping
\begin{eqnarray}\label{eq. g0}
\mathcal{G}^0:C([0,T];L^2(\mu))\rightarrow C([0,T];F^*_{1,2})\cap L^2([0,T];L^2(\mu))
\end{eqnarray}
such that $\mathcal{G}^0(\int_0^\cdot h(s)ds):=Y^{h}(\cdot)$, where $Y^{h}$ is the unique strong solution to \eref{eq:29}.

According to Yamada-Watanabe theorem (cf.\cite{RSZ}) and Theorem \ref{Th 3.1}, there exists a Borel-measurable function
 \begin{eqnarray}\label{eq. g1}
\mathcal{G}^\varepsilon:C([0,T];L^2(\mu))\rightarrow C([0,T];F^*_{1,2})\cap L^2([0,T];L^2(\mu))
\end{eqnarray}
such that
$$X^\varepsilon(\cdot)=\mathcal{G}^\varepsilon(W(\cdot)),$$
where $X^\varepsilon$ is the unique strong solution to \eref{eq:1}.

Denote $\mathcal{A}_M$ as
\begin{eqnarray}\label{eqnarray2}
\mathcal{A}_M=\{h\in \mathcal{A}:h(\omega)\in \mathcal{S}_M, \Bbb{P}\text{-}a.s.\},
\end{eqnarray}
where
\begin{eqnarray*}
\mathcal{A}=\{h:h~\text{is~an}~L^2(\mu)\text{-valued}~\{\mathcal{F}_t\}_{t\geq0}\text{-predictable~process~such~that}~\!\!\int_0^T\!\! |h(s)|_2^2ds<\!\infty,~\Bbb{P}\text{-}a.s.\},
\end{eqnarray*}
and
\begin{eqnarray}\label{eqn8}
\mathcal{S}_M=\{h\in L^2([0,T];L^2(\mu)): \int_0^T|h(s)|_2^2ds\leq M\}.
\end{eqnarray}
Note that the set $\mathcal{S}_M$ endowed with the following metric is a Polish space (complete separatable metric space) \cite{BDM1}: $d_1(h,k)=\sum_{i=1}^\infty \frac{1}{2^i}\Big|\int_0^T\langle h(s)-k(s),\tilde{e_i}(s)\rangle_2ds\Big|$, where $h,k\in \mathcal{S}_M$ and $\{\tilde{e_i}\}_{i=1}^\infty$ is an orthonormal basis for $L^2([0,T];L^2(\mu))$.

According to \cite[Theorem 3.2]{MSZ}, our claim is established once we have proved:

\textbf{(a)}~~For every $M<\infty$, for any family $\{h_\varepsilon\}_{\varepsilon>0}\subset\mathcal{A}_M$ and for any $\delta>0$,
$$\lim_{\varepsilon\rightarrow0}\Bbb{P}(\rho(Y^{h_\varepsilon},X^{h_\varepsilon})>\delta)=0,$$
where $X^{h_\varepsilon}:=\mathcal{G}^\varepsilon(W(\cdot)+\frac{1}{\sqrt{\varepsilon}}\int_0^\cdot h_\varepsilon(s)ds)$, $Y^{h\varepsilon}:=\mathcal{G}^0(\int_0^\cdot h_\varepsilon(s)ds)$, and $\rho(\cdot,\cdot)$ stands for the metric in the space $C([0,T];F^*_{1,2})$.

\vspace{2mm}
\textbf{(b)}~~For every $M<\infty$ and any family $\{h_\varepsilon\}_{\varepsilon>0}\subset\mathcal{S}_M$ that converges to some element $h$ as $\varepsilon\rightarrow0$, $\mathcal{G}^0(\int_0^\cdot h_\varepsilon(s)ds)$ converges to $\mathcal{G}^0(\int_0^\cdot h(s)ds)$ in the space $C([0,T];F^*_{1,2})$.

In Section 5, \textbf{(a)} and \textbf{(b)} will be checked respectively.
\end{proof}
\vspace{2mm}
\smallskip

\section{Proof of Theorem \ref{Th 3.2}}
\setcounter{equation}{0}
 \setcounter{definition}{0}

In this section, we will prove the following result on existence, uniqueness and provide a priori estimates for the solutions of the skeleton equations (\ref{eq:29}), which implies Theorem \ref{Th 3.2}. The a priori estimates are devoted to obtain \textbf{(b)} in the proof of Theorem \ref{Th 3.3}.

\begin{theorem}
Suppose that \textbf{(H1)} and \textbf{(H2)} are satisfied.
Then, for each $x\in L^2(\mu)$, $h\in \mathcal{S}_M$, there exists a unique strong solution $Y^{h}$ to
\eref{eq:29}, and a positive constant $C_{T,M}\in(0,\infty)$ which depends on $T$ and $M$ satisfying
\begin{eqnarray}\label{eqnarray3}
\sup_{h\in\mathcal{S}_M}\sup_{t\in[0, T]}\big|Y^{h}(t)\big|_2^2\leq C_{T,M}(|x|_2^2+1).
\end{eqnarray}
Suppose that \textbf{(H1)}, \textbf{(H2)(i)}, \textbf{(H2)(ii)} and the following are satisfied,
\begin{equation}\label{eq:2.2}
\Psi(r)r\geq c r^2,\ \forall\ r\in \mathbb{R},
\end{equation}
where $c\in(0, \infty)$. Then, for all $x\in F_{1,2}^*$, there is a unique strong solution $Y^{h}$ to
\eref{eq:29}.
\end{theorem}

The proof of Theorem 4.1 is a generalization of \cite[Theorem 3.1]{RWX}, and the main difference is that there is no diffusion term but one more drift term in \eref{eq:29}, so appropriate estimates on the drift term are needed. To prove Theorem 4.1, we need to consider the following approximating equations for \eref{eq:29}:
\begin{eqnarray}\label{eq:30}
dY_\nu^{h}(t)+\big((\nu-L)\Psi(Y^{h}_\nu(t))-B(t,Y^{h}_\nu(t))\circ h(t)\big)dt=0,
\end{eqnarray}
with initial value $Y_\nu^{h}(0)=x\in L^2(\mu)$, where $\nu\in(0,1)$. For \eref{eq:30}, we have the following lemma.
\begin{lemma}
Suppose \textbf{(H1)} and \textbf{(H2)} are satisfied. Then, for each $x\in
L^2(\mu)$, $h\in \mathcal{S}_M$, there is a unique solution to
(\ref{eq:30}), denoted by $Y^{h}_\nu$, i.e., it has the following
properties,
\begin{equation}\label{equ:2.4}
Y^{h}_\nu\in L^2\big([0,T];L^2(\mu)\big)\cap
C([0,T];F_{1,2}^*)
\end{equation}
and
\begin{equation}\label{equ:2.5}
Y^{h}_\nu(t)+(\nu-L)\int_0^t\Psi(Y^{h}_\nu(s))ds=x+\int_0^tB(s,Y^{h}_\nu(s))\circ h(s)ds,\ \forall~ t\in[0,T]
\end{equation}
holds in $F^*_{1,2}$. Furthermore, there exists a positive constant $C_{T,M}\in(0,\infty)$ which depends on $T$ and $M$ but is independent of $\nu\in(0,1)$ such that,
\begin{equation}\label{equ:2.6}
\sup_{h\in\mathcal{S}_M}\sup_{t\in[0, T]}|Y^{h}_\nu(t)|_2^2\leq C_{T,M}(|x|_2^2+1),\ \forall \nu\in(0,1).
\end{equation}
Suppose that \textbf{(H1)}, \textbf{(H2)(i)}, \textbf{(H2)(ii)} and  \eref{eq:2.2} are satisfied, then for
all $x\in F_{1,2}^*$, there is a unique solution
$Y^{h}_\nu$ to (\ref{eq:30}) satisfying (\ref{equ:2.4}) and (\ref{equ:2.5}).
\end{lemma}

\begin{proof}
First, assume that $x\in F^*_{1,2}$ and \eref{eq:2.2} is satisfied. Define
$$
A(t,u):=(L-\nu)\Psi(u)+B(t,u)\circ h(t),
$$
notice that
\begin{eqnarray}\label{eqnarray4}
B(t,u)\circ h(t)\in F^*_{1,2}\subset (L^2(\mu))^*~~\text{since}~~B(t,u)\in L_2(L^2(\mu),F^*_{1,2}),
\end{eqnarray}
consequently,
$$A(t,u):[0,T]\times L^2(\mu)\rightarrow (L^2(\mu))^*.$$
By applying \cite[Theorem 1.1]{LR1} under the Gelfand triple $L^2(\mu)\subset F^*_{1,2}\equiv F_{1,2}\subset (L^2(\mu))^*$, we can prove the existence and uniqueness of the solutions to \eref{eq:30}. Now, we verify the four conditions in \cite[Theorem 1.1]{LR1}. As mentioned above, compared with \cite[Lemma 3.1, Step 1]{RWX}, our $A$ here has one more term $B(t,\cdot)\circ h(t)$, so we only need to estimate terms with $B(t,\cdot)\circ h(t)$.

\vspace{2mm}
\textbf{(i) }\textbf{Hemicontinuity}
\vspace{2mm}

Let $u,v,w\in V(:=L^2(\mu))$. We need to show that for $\lambda\in \Bbb{R},~ |\lambda|\leq1$,
$$
\lim_{\lambda\rightarrow0}~_{(L^2(\mu))^*}\big\langle A(t,u+\lambda v),w\big\rangle_{L^2(\mu)}-~ _{(L^2(\mu))^*}\big\langle A(t,u),w\big\rangle_{L^2(\mu)}=0.
$$
Since in \cite[ Lemma 3.1, Step 1]{RWX}, the authors have proved that
$$\lim_{\lambda\rightarrow0}~ _{(L^2(\mu))^*}\big\langle (L-\nu)\Psi(u+\lambda v),w\big\rangle_{L^2(\mu)}-~ _{(L^2(\mu))^*}\big\langle (L-\nu)\Psi(u),w\big\rangle_{L^2(\mu)}=0,$$
here we only need to prove
\begin{eqnarray}\label{eq:3.1}
\lim_{\lambda\rightarrow0}~ _{(L^2(\mu))^*}\big\langle B(t,u+\lambda v)\circ h(t),w\big\rangle_{L^2(\mu)}-~ _{(L^2(\mu))^*}\big\langle B(t,u)\circ h(t),w\big\rangle_{L^2(\mu)}=0.
\end{eqnarray}
Notice that by \eref{eqn6}, \eref{eqnarray4} and \textbf{(H2)(i)},
\begin{eqnarray*}
&&\!\!\!\!\!\!\!\! _{(L^2(\mu))^*}\langle (B(t,u+\lambda v)-B(t,u))\circ h(t), w\rangle_{L^2(\mu)}\nonumber\\
=&&\!\!\!\!\!\!\!\!\langle(B(t,u+\lambda v)-B(t,u))\circ h(t), w\rangle_{F^*_{1,2}}\nonumber\\
\leq&&\!\!\!\!\!\!\!\!\|B(t,u+\lambda v)-B(t,u)\|_{L_2(L^2(\mu),F^*_{1,2})}\cdot |h(t)|_2 \cdot \|w\|_{F^*_{1,2}}\nonumber\\
\leq&&\!\!\!\!\!\!\!\!C\lambda\|v\|_{F^*_{1,2}}\cdot|h(t)|_2\cdot\|w\|_{F^*_{1,2}}\rightarrow0,~~\text{as}~~\lambda\rightarrow0,
\end{eqnarray*}
which implies \eref{eq:3.1}.

\vspace{2mm}
\textbf{(ii)} \textbf{Local Monotonicity}
\vspace{2mm}

Let $u,v\in V(=L^2(\mu))$. By \eref{eqn6}, \eref{eqnarray4} and \textbf{(H2)(i)}, we know that
\begin{eqnarray}\label{4}
&&\!\!\!\!\!\!\!\!_{(L^2(\mu))^*}\big\langle(B(t,u)-B(t,v))\circ h(t),u-v\big\rangle_{L^2(\mu)}\nonumber\\
=&&\!\!\!\!\!\!\!\!\big\langle(B(t,u)-B(t,v))\circ h(t),u-v\big\rangle_{F^*_{1,2}}\nonumber\\
\leq&&\!\!\!\!\!\!\!\!\|B(t,u)-B(t,v)\|_{L_2(L^2(\mu),F^*_{1,2})}\cdot |h(t)|_2\cdot\|u-v\|_{F^*_{1,2}}\nonumber\\
\leq&&\!\!\!\!\!\!\!\!C_1\|u-v\|_{F^*_{1,2}}\cdot |h(t)|_2\cdot\|u-v\|_{F^*_{1,2}}\nonumber\\
=&&\!\!\!\!\!\!\!\!C_1|h(t)|_2\cdot\|u-v\|^2_{F^*_{1,2}}.
\end{eqnarray}
Denote by $Lip\Psi$ the Lipschitz constant of $\Psi$. From \cite[Lemma 3.1]{RWX}, we know that
\begin{eqnarray*}
&&\!\!\!\!\!\!\!\!_{(L^2(\mu))^*}\big\langle(L-\nu)(\Psi(u)-\Psi(v)), u-v\big\rangle_{L^2(\mu)}\nonumber\\
\leq&&\!\!\!\!\!\!\!\!\big(\frac{(1-\nu)^2}{\widetilde{\alpha}}\big)\cdot \|u-v\|^2_{F^*_{1,2}},
\end{eqnarray*}
where
\begin{eqnarray}\label{eqnarray6}
\widetilde{\alpha}:=(k+1)^{-1},~~ k:=Lip\Psi.
\end{eqnarray}
So
\begin{eqnarray}\label{eq:3.2}
&&\!\!\!\!\!\!\!\!_{(L^2(\mu))^*}\big\langle A(t,u)-A(t,v),u-v\big\rangle_{L^2(\mu)}\nonumber\\
\leq&&\!\!\!\!\!\!\!\!\Big(\frac{(1-\nu)^2}{\widetilde{\alpha}}+C_1|h(t)|_2\Big)\cdot \|u-v\|^2_{F^*_{1,2}},
\end{eqnarray}
which implies the local monotonicity.

\vspace{2mm}
\textbf{(iii)} \textbf{Coercivity}
\vspace{2mm}

Let $u\in V(=L^2(\mu))$. By \eref{eqn6}, \eref{eqnarray4} and \textbf{(H2)(ii)}, we have
\begin{eqnarray}
&&\!\!\!\!\!\!\!\!_{(L^2(\mu))^*}\big\langle B(t,u)\circ h(t),u\big\rangle_{L^2(\mu)}\nonumber\\
=&&\!\!\!\!\!\!\!\!\big\langle B(t,u)\circ h(t),u\big\rangle_{F^*_{1,2}}\nonumber\\
\leq&&\!\!\!\!\!\!\!\!\|B(t,u)\|_{L_2(L^2(\mu),F^*_{1,2})}\cdot|h(t)|_2\cdot\|u\|_{F^*_{1,2}}\nonumber\\
\leq&&\!\!\!\!\!\!\!\!2C_2|h(t)|_2\cdot(\|u\|^2_{F^*_{1,2}}+1).
\end{eqnarray}
From \cite[Lemma 3.1]{RWX}, we know that, for any $\theta>0$,
\begin{eqnarray*}
&&\!\!\!\!\!\!\!\!_{(L^2(\mu))^*}\big\langle (1-L)(\Psi(u)),u\big\rangle_{L^2(\mu)}\nonumber\\
\leq&&\!\!\!\!\!\!\!\!\Big[-c+\theta^2k^2(1-\nu)\Big]\cdot|u|_2^2+\frac{(1-\nu)}{\theta^2}\cdot\|u\|^2_{F^*_{1,2}}.
\end{eqnarray*}
So, we obtain
\begin{eqnarray}\label{eq:3.3}
&&\!\!\!\!\!\!\!\!_{(L^2(\mu))^*}\big\langle A(t,u),u\big\rangle_{L^2(\mu)}\nonumber\\
\leq&&\!\!\!\!\!\!\!\!\Big[-c+\theta^2k^2(1-\nu)\Big]\cdot|u|_2^2+\Big[\frac{(1-\nu)}{\theta^2}+2C_2|h(t)|_2\Big]\cdot(\|u\|^2_{F^*_{1,2}}+1).
\end{eqnarray}
Choosing $\theta$ small enough, $-c+\theta^2k^2(1-\nu)$ becomes
negative, which implies the coercivity.

\vspace{2mm}
\textbf{(iv) Growth}

Let $u\in V(=L^2(\mu))$. Notice that
\begin{eqnarray*}
\|A(t,u)\|_{(L^2(\mu))^*}=\sup_{|v|_2=1}~_{(L^2(\mu))^*}\big\langle(L-\nu)(\Psi(u))+B(t,u)\circ h(t),v\big\rangle_{L^2(\mu)}.
\end{eqnarray*}
From \cite[Lemma 3.1]{RWX}, we know that
$$\|(L-\nu)\Psi(u)\|_{(L^2(\mu))^*}\leq2k|u|_2.$$
Since $L^2(\mu)\subset F^*_{1,2}\subset (L^2(\mu))^*$ continuously and densely, from \textbf{(H2)(ii)} we get
$$_{(L^2(\mu))^*}\big\langle B(t,u)\circ h(t),v\big\rangle_{L^2(\mu)}\leq C_2|h(t)|_2\cdot(|u|_2+1)\cdot|v|_2.$$
So
\begin{eqnarray}\label{eq:3.4}
\|A(t,u)\|_{(L^2(\mu))^*}\leq\big(2k+C_2|h(t)|_2\big)\cdot(|u|_2+1).
\end{eqnarray}
Hence the growth holds.

Then by \cite[Theorem 1.1]{LR1} or \cite[Theorem 4.2.4]{PR}, there exists a unique solution to \eref{eq:30}, denoted by $Y^{h}_\nu$, which takes value in $F^*_{1,2}$ and satisfies (\ref{equ:2.4}) and (\ref{equ:2.5}).

\begin{remark}
As shown above, the coefficients in the right-hand sides of \eref{eq:3.2}, \eref{eq:3.3} and \eref{eq:3.4} have a term $|h(t)|_2$, which looks different from the conditions in \cite[Theorem 1.1]{LR1}, where the coefficients are constants. However,
since $h\in\mathcal{S}_M$, i.e., $\int_{0}^{T}|h(s)|_2^2ds\leq M$, by using the same idea as in the proof of \cite[Theorem 1.1]{LR1}, it is not difficult to get Lemma 4.1. So, here we still regard (i)-(iv) as the corresponding conditions of \cite[Theorem 1.1]{LR1}.
\end{remark}

\vspace{2mm}
If $x\in F^*_{1,2}$, but \eref{eq:2.2} is not satisfied, (i), (ii) and (iv) still hold, but (iii) may not be true in general. We shall approximate $\Psi$ by $\Psi+\lambda I$, $\lambda\in(0,1)$, i.e., we consider the following approximating equations:
\begin{eqnarray}\label{eq:2.3}
dY^{h}_{\nu,\lambda}(t)=(L-\nu)(\Psi(Y^{h}_{\nu,\lambda}(t))dt+\lambda Y^{h}_{\nu,\lambda}(t))dt+B(t,Y^{h}_{\nu,\lambda}(t))\circ h(t)dt,
\end{eqnarray}
with initial value $Y^{h}_{\nu,\lambda}(0))=x$. By \cite[Theorem 1.1]{LR1} and Remark 4.1, we know that if $x\in F^*_{1,2}$, there exists a unique solution $Y^{h}_{\nu,\lambda}$ to \eref{eq:2.3} such that $Y^{h}_{\nu,\lambda}\in L^2\big([0,T];L^2(\mu)\big)\cap C([0,T];F_{1,2}^*)$, and for all $t\in[0,T]$,
\begin{eqnarray}\label{eq00}
Y^{h}_{\nu,\lambda}(t)=x+(L-\nu)\int_0^t(\Psi(Y^{h}_{\nu,\lambda}(s))+\lambda Y^{h}_{\nu,\lambda}(s))ds+\int_0^tB(s,Y^{h}_{\nu,\lambda}(s))\circ h(s)ds
\end{eqnarray}
holds in $F^*_{1,2}$, and
\begin{eqnarray}\label{eqna2.1}
\sup_{t\in[0,T]}\|Y^{h}_{\nu,\lambda}(t)\|^2_{F^*_{1,2}}<\infty.
\end{eqnarray}
Next, we want to prove that the sequence $\{Y^{h}_{\nu,\lambda}\}$
converges to the solution of \eref{eq:30} as $\lambda\to 0$. From now on,
we assume that the initial value $x\in L^2(\mu)$, and we have the following result for \eref{eq:2.3}.
\begin{claim}\label{claim1}
\begin{eqnarray*}
\sup_{s\in[0,t]}|Y^{h}_{\nu,\lambda}(s)|_2^2+4\lambda \int_0^t\|Y^{h}_{\nu,\lambda}(s)\|^2_{F_{1,2}}ds\leq C_{M,T}(|x|_2^2+1), \forall\nu, \lambda\in(0,1),~h\in\mathcal{S}_M,~t\in[0,T],
\end{eqnarray*}
and $Y^{h}_{\nu,\lambda}\in C([0,T];L^2(\mu))$. Here $C_{T,M}$ is a positive constant which depends on $T$ and $M$, but is independent of $\nu, \lambda\in(0,1)$.
\end{claim}
\begin{proof}
Rewrite \eref{eq:2.3} in the following form
\begin{eqnarray*}
Y^{h}_{\nu,\lambda}(t)=&&\!\!\!\!\!\!\!\!x+\int_0^t(L-\nu)(\Psi(Y^{h}_{\nu,\lambda}(s))+\lambda Y^{h}_{\nu,\lambda}(s))ds\nonumber\\
&&\!\!\!\!\!\!\!\!+\int_0^tB(s,Y^{h}_{\nu,\lambda}(s))\circ h(s)ds, ~~\forall t\in[0,T].
\end{eqnarray*}
Let $\nu<\alpha<\infty$, applying the operator $(\alpha-L)^{-\frac{1}{2}}:F^*_{1,2}\rightarrow L^2(\mu)$ to both sides of the above equation, we get
\begin{eqnarray*}
(\alpha-L)^{-\frac{1}{2}}Y^{h}_{\nu,\lambda}(t)=&&\!\!\!\!\!\!\!\!(\alpha-L)^{-\frac{1}{2}}x+\int_0^t(L-\nu)(\alpha-L)^{-\frac{1}{2}}(\Psi(Y^{h}_{\nu,\lambda}(s))+\lambda Y^{h}_{\nu,\lambda}(s))ds\nonumber\\
&&\!\!\!\!\!\!\!\!+\int_0^t(\alpha-L)^{-\frac{1}{2}}\big(B(s,Y^{h}_{\nu,\lambda}(s))\circ h(s)\big)ds.
\end{eqnarray*}
Applying the chain rule in $L^2(\mu)$, we obtain that for $t\in[0,T]$,
\begin{eqnarray}\label{eq:2.7}
&&\!\!\!\!\!\!\!\!\big|(\alpha-L)^{-\frac{1}{2}}Y^{h}_{\nu,\lambda}(t)\big|_2^2\nonumber\\
=&&\!\!\!\!\!\!\!\!|(\alpha-L)^{-\frac{1}{2}}x|_2^2+2\int_0^t~_{F^*_{1,2}}\Big\langle (L-\nu)(\alpha-L)^{-\frac{1}{2}}\big(\Psi(Y^{h}_{\nu,\lambda}(s))\big),(\alpha-L)^{-\frac{1}{2}}Y^{h}_{\nu,\lambda}(s)\Big\rangle_{F_{1,2}}ds\nonumber\\
&&\!\!\!\!\!\!\!\!+2\lambda\int_0^t~_{F^*_{1,2}}\Big\langle(L-\nu)(\alpha-L)^{-\frac{1}{2}}Y^{h}_{\nu,\lambda}(s),(\alpha-L)^{-\frac{1}{2}}Y^{h}_{\nu,\lambda}(s)\Big\rangle_{F_{1,2}}ds\nonumber\\
&&\!\!\!\!\!\!\!\!+2\int_0^t~_{F^*_{1,2}}\Big\langle(\alpha-L)^{-\frac{1}{2}}\big(B(s,Y^{h}_{\nu,\lambda}(s))\circ h(s)\big),(\alpha-L)^{-\frac{1}{2}}Y^{h}_{\nu,\lambda}(s)\Big\rangle_{F_{1,2}}ds.
\end{eqnarray}
From \cite[(3.19), (3.20)]{RWX2022}, we know that
\begin{eqnarray}\label{eq:2.6}
2\int_0^t~_{F^*_{1,2}}\Big\langle (L-\nu)(\alpha-L)^{-\frac{1}{2}}\big(\Psi(Y^{h}_{\nu,\lambda}(s))\big),(\alpha-L)^{-\frac{1}{2}}Y^{h}_{\nu,\lambda}(s)\Big\rangle_{F_{1,2}}ds\leq0,
\end{eqnarray}
and
\begin{eqnarray}\label{eq:2.5}
&&\!\!\!\!\!\!\!\!2\lambda\int_0^t~_{F^*_{1,2}}\Big\langle(L-\nu)(\alpha-L)^{-\frac{1}{2}}Y^{h}_{\nu,\lambda}(s),(\alpha-L)^{-\frac{1}{2}}Y^{h}_{\nu,\lambda}(s)\Big\rangle_{F_{1,2}}ds\nonumber\\
\leq&&\!\!\!\!\!\!\!\!-2\lambda\int_0^t\|(\alpha-L)^{-\frac{1}{2}}Y^{h}_{\nu,\lambda}(s)\|^2_{F_{1,2}}ds+2\int_0^t|(\alpha-L)^{-\frac{1}{2}}Y^{h}_{\nu,\lambda}(s)|^2_2ds.
\end{eqnarray}
To estimate the fourth term in the right-hand side of \eref{eq:2.7}, we first recall the Gelfand triple $F_{1,2}\subset L^2(\mu)\equiv (L^2(\mu))^*\subset F^*_{1,2}$, the Riesz map which identifies $L^2(\mu)$ and $(L^2(\mu))^*$ is the identity map $I:L^2(\mu)\rightarrow (L^2(\mu))^*$, i.e., $\forall~u\in L^2(\mu)$, $I(u)=u$, and in fact $(L^2(\mu))^*=L^2(\mu)$. Notice that from \eref{eqnarray4} we know $(\alpha-L)^{-\frac{1}{2}}\big(B(s,Y^{h}_{\nu,\lambda}(s))\circ h(s)\big)$ takes value in the space $L^2(\mu)$, so by \eref{eqn6} we have
\begin{eqnarray}\label{eq:2.4}
&&\!\!\!\!\!\!\!\!_{F^*_{1,2}}\Big\langle(\alpha-L)^{-\frac{1}{2}}\big(B(s,Y^{h}_{\nu,\lambda}(s))\circ h(s)\big),(\alpha-L)^{-\frac{1}{2}}Y^{h}_{\nu,\lambda}(s)\Big\rangle_{F_{1,2}}\nonumber\\
=&&\!\!\!\!\!\!\!\!\Big\langle(\alpha-L)^{-\frac{1}{2}}\big(B(s,Y^{h}_{\nu,\lambda}(s))\circ h(s)\big),(\alpha-L)^{-\frac{1}{2}}Y^{h}_{\nu,\lambda}(s)\Big\rangle_2.
\end{eqnarray}
Multiplying both sides of \eref{eq:2.7} by $\alpha$, taking \eref{eq:2.6}-\eref{eq:2.4} into account, since $\sqrt{\alpha}(\alpha-L)^{-\frac{1}{2}}$ is a contraction on $L^2(\mu)$, by \textbf{(H2)(iii)}, \eref{eq:2.7} yields that for all $t\in[0,T]$,
\begin{eqnarray}\label{eq:2.8}
&&\!\!\!\!\!\!\!\!\Big|\sqrt{\alpha}(\alpha-L)^{-\frac{1}{2}}Y^{h}_{\nu,\lambda}(t)\Big|_2^2+2\lambda \int_0^t\|\sqrt{\alpha}(\alpha-L)^{-\frac{1}{2}}Y^{h}_{\nu,\lambda}(s)\|^2_{F_{1,2}}ds\nonumber\\
\leq&&\!\!\!\!\!\!\!\!|\sqrt{\alpha}(\alpha-L)^{-\frac{1}{2}}x|_2^2+2\int_0^t\Big|\sqrt{\alpha}(\alpha-L)^{-\frac{1}{2}}\big(B(s,Y^{h}_{\nu,\lambda}(s))\circ h(s)\big)\Big|_2\cdot|\sqrt{\alpha}(\alpha-L)^{-\frac{1}{2}}Y^{h}_{\nu,\lambda}(s)|_2ds\nonumber\\
&&\!\!\!\!\!\!\!\!+2\int_0^t|(\alpha-L)^{-\frac{1}{2}}Y^{h}_{\nu,\lambda}(s)|^2_2ds\nonumber\\
\leq&&\!\!\!\!\!\!\!\!|\sqrt{\alpha}(\alpha-L)^{-\frac{1}{2}}x|_2^2+2\int_0^t|Y^{h}_{\nu,\lambda}(s)|^2_2ds\nonumber\\
&&\!\!\!\!\!\!\!\!+2\int_0^t\big\|\sqrt{\alpha}(\alpha-L)^{-\frac{1}{2}}B(s,Y^{h}_{\nu,\lambda}(s))\big\|_{L_2(L^2(\mu),L^2(\mu))}\cdot|h(s)|_2\cdot|\sqrt{\alpha}(\alpha-L)^{-\frac{1}{2}}Y^{h}_{\nu,\lambda}(s)|_2ds\nonumber\\
\leq&&\!\!\!\!\!\!\!\!|\sqrt{\alpha}(\alpha-L)^{-\frac{1}{2}}x|_2^2+2\int_0^t|Y^{h}_{\nu,\lambda}(s)|^2_2ds\nonumber\\
&&\!\!\!\!\!\!\!\!+2\int_0^tC_3(|Y^{h}_{\nu,\lambda}(s)|_{2}+1)\cdot|h(s)|_2\cdot|\sqrt{\alpha}(\alpha-L)^{-\frac{1}{2}}Y^{h}_{\nu,\lambda}(s)|_2ds.
\end{eqnarray}
By Young's inequality, we get
\begin{eqnarray*}
&&\!\!\!\!\!\!\!\!\sup_{s\in[0,t]}\big|\sqrt{\alpha}(\alpha-L)^{-\frac{1}{2}}Y^{h}_{\nu,\lambda}(s)\big|_2^2+2\lambda\int_0^t\|\sqrt{\alpha}(\alpha-L)^{-\frac{1}{2}}Y^{h}_{\nu,\lambda}(s)\|^2_{F_{1,2}}ds\nonumber\\
\leq&&\!\!\!\!\!\!\!\!|\sqrt{\alpha}(\alpha-L)^{-\frac{1}{2}}x|_2^2+2\int_0^t|Y^{h}_{\nu,\lambda}(s)|^2_2ds\nonumber\\
&&\!\!\!\!\!\!\!\!+\frac{1}{2T}\int_0^t|\sqrt{\alpha}(\alpha-L)^{-\frac{1}{2}}Y^{h}_{\nu,\lambda}(s)|_2^2ds+2C_3^2T\int_0^t|h(s)|_2^2\cdot(|Y^{h}_{\nu,\lambda}(s)|^2_{2}+1)ds\nonumber\\
\leq&&\!\!\!\!\!\!\!\!|\sqrt{\alpha}(\alpha-L)^{-\frac{1}{2}}x|_2^2+\frac{1}{2}\sup_{s\in[0,t]}|\sqrt{\alpha}(\alpha-L)^{-\frac{1}{2}}Y^{h}_{\nu,\lambda}(s)|_2^2\nonumber\\
&&\!\!\!\!\!\!\!\!+\int_0^t\big(2C_3^2T|h(s)|_2^2+2\big)|Y^{h}_{\nu,\lambda}(s)|^2_{2}ds+2C_3^2T\int_0^t|h(s)|_2^2ds.
\end{eqnarray*}
Since $|\sqrt{\alpha}(\alpha-L)^{-\frac{1}{2}}\cdot|_2^2$ is equivalent to $\|\cdot\|_{F^*_{1,2}}$, by \eqref{eqna2.1} we have that for all $t\in[0,T]$,
\begin{eqnarray*}
&&\!\!\!\!\!\!\!\!\sup_{s\in[0,t]}\big|\sqrt{\alpha}(\alpha-L)^{-\frac{1}{2}}Y^{h}_{\nu,\lambda}(s)\big|_2^2+4\lambda\int_0^t\|\sqrt{\alpha}(\alpha-L)^{-\frac{1}{2}}Y^{h}_{\nu,\lambda}(s)\|^2_{F_{1,2}}ds\nonumber\\
\leq&&\!\!\!\!\!\!\!\!2|x|_2^2+4\int_0^t\big(C_3^2T|h(s)|_2^2+1\big)|Y^{h}_{\nu,\lambda}(s)|^2_{2}ds+4C_3^2T\int_0^t|h(s)|_2^2ds,
\end{eqnarray*}
letting $\alpha\rightarrow\infty$, by \cite[(3.23)]{RWX2022}, \cite[(3.21)]{RWX2022} and Fatou's lemma, we have that for all $t\in[0,T]$,
\begin{eqnarray*}
&&\!\!\!\!\!\!\!\!\sup_{s\in[0,t]}\big|Y^{h}_{\nu,\lambda}(s)\big|_2^2+4\lambda\int_0^t\|Y^{h}_{\nu,\lambda}(s)\|^2_{F_{1,2}}ds\nonumber\\
\leq&&\!\!\!\!\!\!\!\!\liminf_{\alpha\rightarrow\infty}\Big[\sup_{s\in[0,t]}\big|\sqrt{\alpha}(\alpha-L)^{-\frac{1}{2}}Y^{h}_{\nu,\lambda}(s)\big|_2^2+4\lambda\int_0^t\|\sqrt{\alpha}(\alpha-L)^{-\frac{1}{2}}Y^{h}_{\nu,\lambda}(s)\|^2_{F_{1,2}}ds\Big]\nonumber\\
\leq&&\!\!\!\!\!\!\!\!2|x|_2^2+4\int_0^t\big(C_3^2T|h(s)|_2^2+1\big)|Y^{h}_{\nu,\lambda}(s)|^2_{2}ds+4C_3^2T\int_0^t|h(s)|_2^2ds.
\end{eqnarray*}
Then by Gronwall's lemma, we know that for all $\nu, \lambda\in(0,1)$, $h\in\mathcal{S}_M$, $t\in[0,T]$,
\begin{eqnarray*}
&&\!\!\!\!\!\!\!\!\sup_{s\in[0,t]}|Y^{h}_{\nu,\lambda}(s)|_2^2+4\lambda\int_0^t\|Y^{h}_{\nu,\lambda}(s)\|^2_{F_{1,2}}ds\nonumber\\
\leq&&\!\!\!\!\!\!\!\! \big(2|x|_2^2+4C_3^2T\int_0^t|h(s)|_2^2ds\big)\cdot e^{4\int_0^tC_3^2T|h(s)|_2^2+1ds}\nonumber\\
\leq&&\!\!\!\!\!\!\!\!\big(2|x|_2^2+4C_3^2T\cdot M\big)\cdot e^{4C_3^2TM+4T}\nonumber\\
:=&&\!\!\!\!\!\!\!\! C_{T,M}(|x|_2^2+1).
\end{eqnarray*}\hspace{\fill}$\Box$
\end{proof}

\begin{claim}\label{claim2}
$\{Y^{h}_{\nu,\lambda}\}_{\lambda\in(0,1)}$ converges to an element $Y^{h}_\nu\in L^2([0,T]; L^2(\mu))$ as $\lambda\rightarrow0$.
\end{claim}
\begin{proof}
By the chain rule, we get that for $\lambda, \lambda'\in(0,1)$ and $t\in [0,T]$,
\begin{eqnarray}\label{eqn1}
&&\!\!\!\!\!\!\!\!\|Y^{h}_{\nu,\lambda}(t)-Y^{h}_{\nu,\lambda'}(t)\|^2_{F^*_{1,2,\nu}}\nonumber\\
&&\!\!\!\!\!\!\!\!+2\int_0^t\Big\langle(\nu-L)\big(\Psi(Y^{h}_{\nu,\lambda}(s))-\Psi(Y^{h}_{\nu,\lambda'}(s))+\lambda Y^{h}_{\nu,\lambda}(s)-\lambda'Y^{h}_{\nu,\lambda'}(s)\big)\Big\rangle_{F^*_{1,2,\nu}}ds\nonumber\\
=&&\!\!\!\!\!\!\!\!2\int_0^t\Big\langle Y^{h}_{\nu,\lambda}(s)-Y^{h}_{\nu,\lambda'}(s),\big(B(s,Y^{h}_{\nu,\lambda}(s))-B(s,Y^{h}_{\nu,\lambda'}(s))\big)\circ h(s)\Big\rangle_{F^*_{1,2,\nu}}ds.
\end{eqnarray}
To estimate the second term in the left-hand side of \eref{eqn1}, first by \eref{eqnarray5}, we know that
\begin{eqnarray}\label{eq0}
&&\!\!\!\!\!\!\!\!\Big\langle(\nu-L)\big(\Psi(Y^{h}_{\nu,\lambda}(s))-\Psi(Y^{h}_{\nu,\lambda'}(s))+\lambda Y^{h}_{\nu,\lambda}(s)-\lambda'Y^{h}_{\nu,\lambda'}(s)\big)\Big\rangle_{F^*_{1,2,\nu}}\nonumber\\
=&&\!\!\!\!\!\!\!\!\Big\langle\Psi(Y^{h}_{\nu,\lambda}(s))-\Psi(Y^{h}_{\nu,\lambda'}(s))+\lambda Y^{h}_{\nu,\lambda}(s)-\lambda'Y^{h}_{\nu,\lambda'}(s),Y^{h}_{\nu,\lambda}(s)-Y^{h}_{\nu,\lambda'}(s)\Big\rangle_2,
\end{eqnarray}
besides,
\begin{eqnarray}\label{eqn1.1}
(\Psi(r)-\Psi(r'))(r-r')\geq\widetilde{\alpha}\cdot|\Psi(r)-\Psi(r')|^2,~\forall~r,~r'\in\Bbb{R},
\end{eqnarray}
where $\widetilde{\alpha}:=(Lip \Psi+1)^{-1}$, so we have
\begin{eqnarray}\label{eqn2}
&&\!\!\!\!\!\!\!\!2\int_0^t\Big\langle\Psi(Y^{h}_{\nu,\lambda}(s))-\Psi(Y^{h}_{\nu,\lambda'}(s))+\lambda Y^{h}_{\nu,\lambda}(s)-\lambda'Y^{h}_{\nu,\lambda'}(s),Y^{h}_{\nu,\lambda}(s)-Y^{h}_{\nu,\lambda'}(s)\Big\rangle_2ds\nonumber\\
\geq&&\!\!\!\!\!\!\!\!2\widetilde{\alpha}\int_0^t|\Psi(Y^{h}_{\nu,\lambda}(s))-\Psi(Y^{h}_{\nu,\lambda'}(s))|_2^2ds\nonumber\\
&&\!\!\!\!\!\!\!\!+2\int_0^t\Big\langle\lambda Y^{h}_{\nu,\lambda}(s)-\lambda'Y^{h}_{\nu,\lambda'}(s),Y^{h}_{\nu,\lambda}(s)-Y^{h}_{\nu,\lambda'}(s)\Big\rangle_2ds.
\end{eqnarray}
For the right-hand side term of \eref{eqn1}, by \textbf{(H2)(i)} and Young's inequality, we know
\begin{eqnarray}\label{eqn3}
&&\!\!\!\!\!\!\!\!2\int_0^t\Big\langle Y^{h}_{\nu,\lambda}(s)-Y^{h}_{\nu,\lambda'}(s),\big(B(s,Y^{h}_{\nu,\lambda}(s))-B(s,Y^{h}_{\nu,\lambda'}(s))\big)\circ h(s)\Big\rangle_{F^*_{1,2,\nu}}ds\nonumber\\
\leq&&\!\!\!\!\!\!\!\!2\int_0^t\|Y^{h}_{\nu,\lambda}(s)-Y^{h}_{\nu,\lambda'}(s)\|_{F^*_{1,2,\nu}}\cdot\|B(s,Y^{h}_{\nu,\lambda}(s))-B(s,Y^{h}_{\nu,\lambda'}(s))\|_{L_2(L^2(\mu),F^*_{1,2,\nu})}\cdot|h(s)|_2ds\nonumber\\
\leq&&\!\!\!\!\!\!\!\!\frac{1}{2}\sup_{s\in[0,t]}\|Y^{h}_{\nu,\lambda}(s)-Y^{h}_{\nu,\lambda'}(s)\|^2_{F^*_{1,2,\nu}}ds+2C_2^2\int_0^t\! T\cdot|h(s)|_2^2\cdot\|Y^{h}_{\nu,\lambda}(s)-Y^{h}_{\nu,\lambda'}(s)\|^2_{F^*_{1,2,\nu}}\!ds.
\end{eqnarray}
Plugging \eref{eqn2} and \eref{eqn3} into \eref{eqn1}, we get
\begin{eqnarray*}
&&\!\!\!\!\!\!\!\!\sup_{s\in[0,t]}\|Y^{h}_{\nu,\lambda}(s)-Y^{h}_{\nu,\lambda'}(s)\|^2_{F^*_{1,2}}+2\widetilde{\alpha}\int_0^t|\Psi(Y^{h}_{\nu,\lambda}(s))-\Psi(Y^{h}_{\nu,\lambda'}(s))|_2^2ds\nonumber\\
\leq&&\!\!\!\!\!\!\!\!\frac{1}{2}\sup_{s\in[0,t]}\|Y^{h}_{\nu,\lambda}(s)-Y^{h}_{\nu,\lambda'}(s)\|^2_{F^*_{1,2,\nu}}+2C_2^2\int_0^t T\cdot|h(s)|^2_2\cdot\|Y^{h}_{\nu,\lambda}(s)-Y^{h}_{\nu,\lambda'}(s)\|^2_{F^*_{1,2,\nu}}ds\nonumber\\
&&\!\!\!\!\!\!\!\!+4(\lambda+\lambda')\int_0^t\big(|Y^{h}_{\nu,\lambda}(s)|_2^2+|Y^{h}_{\nu,\lambda'}(s)|_2^2\big)ds.
\end{eqnarray*}
By Claim 4.1 and Gronwall's lemma, we know that for some constant $C_{T,M}\in(0,\infty)$, which depends on $T$ and $M$, but is independent of $\lambda, \lambda'$, the following inequality holds, i.e.,
\begin{eqnarray*}
&&\!\!\!\!\!\!\!\!\sup_{s\in[0,T]}\|Y^{h}_{\nu,\lambda}(s)-Y^{h}_{\nu,\lambda'}(s)\|^2_{F^*_{1,2}}+4\widetilde{\alpha}\int_0^T|\Psi(Y^{h}_{\nu,\lambda}(s))-\Psi(Y^{h}_{\nu,\lambda'}(s))|_2^2ds\nonumber\\
\leq&&\!\!\!\!\!\!\!\!e^{4C_2^2T\int_{0}^{T}|h(s)|_2^2ds}\cdot8(\lambda+\lambda')\cdot 2C_{T,M}(|x|_2^2+1)\nonumber\\
:=&&\!\!\!\!\!\!\!\!C_{T,M}(\lambda+\lambda')(|x|_2^2+1).
\end{eqnarray*}
Obviously, there exists an element $Y^{h}_\nu$ such that $\{Y^{h}_\nu(t)\}_{t\in [0,T]}$ converges to $Y^{h}_\nu$ in $C([0,T];F^*_{1,2})$. In addition, by Claim 4.1, we know that $Y^{h}_\nu\in L^2([0,T];L^2(\mu))$.\hspace{\fill}$\Box$
\end{proof}

\begin{claim}\label{claim3}
$Y^{h}_\nu$ satisfies \eref{eq:30}.
\end{claim}
\begin{proof}
From Claim 4.2, we know that
\begin{eqnarray}\label{eq01}
Y^{h}_{\nu,\lambda}\rightarrow Y^{h}_\nu,~~\text{in}~~C([0,T];F^*_{1,2})~~\text{as}~\lambda\rightarrow0.
\end{eqnarray}
Since $h\in\mathcal{S}_M$, then by \textbf{(H2)(i)} it is easy to prove
\begin{eqnarray}\label{eq02}
\int_0^\cdot B(s,Y^{h}_{\nu,h}(s))\circ h(s)ds\rightarrow\int_0^\cdot B(s,Y^{h}_\nu(s))\circ h(s)ds,~\text{as}~\lambda\rightarrow0,
\end{eqnarray}
in $C([0,T];F^*_{1,2})$.
From \eref{eq00}, we have
\begin{eqnarray}\label{eq03}
&&\!\!\!\!\!\!\!\!(L-\nu)\int_0^\cdot(\Psi(Y^{h}_{\nu,\lambda}(s))+\lambda Y^{h}_{\nu,\lambda}(s))ds=-x +Y^{h}_{\nu,\lambda}(\cdot)-\int_0^\cdot B(s,Y^{h}_{\nu,\lambda}(s))\circ h(s)ds.
\end{eqnarray}
Then, by \eref{eq01}-\eref{eq03}, we know that
$$\int_0^\cdot\Psi(Y^{h}_{\nu,\lambda}(s))+\lambda Y^{h}_{\nu,\lambda}(s)ds$$
converges to some element in
$$C([0,T];F_{1,2}) ~~\text{as}~~ \lambda\rightarrow0.$$
On the other hand, by Claim 4.1 and Claim 4.2, we know that as $\lambda\rightarrow0$,
$$
\int_0^\cdot\Psi(Y^{h}_{\nu,\lambda}(s))+\lambda Y^{h}_{\nu,\lambda}(s)ds\rightarrow \int_0^\cdot\Psi(Y^{h}_\nu(s))ds
$$
in $L^2([0,T];L^2(\mu))$. This implies  Claim 4.3.\hspace{\fill}$\Box$
\end{proof}

By lower semi-continuity of the norms and Claim 4.1, we also know that \eref{equ:2.6} holds, i.e.,
\begin{eqnarray}\label{eq:32}
\sup_{h\in\mathcal{S}_M}\sup_{s\in[0,T]}|Y^{h}_\nu(s)|_2^2\leq C_{T,M}(|x|_2^2+1).
\end{eqnarray}

\textbf{Uniqueness}

\vspace{2mm}
Assume that $Y^{h}_{\nu,1}$, $Y^{h}_{\nu,2}$ are two solutions to \eref{eq:30}, $\forall~ t\in[0,T]$, we have that
\begin{eqnarray}\label{eq:33}
&&\!\!\!\!\!\!\!\!Y^{h}_{\nu,1}(t)-Y^{h}_{\nu,2}(t)+(\nu-L)\int_0^t\big(\Psi(Y^{h}_{\nu,1}(s))-\Psi(Y^{h}_{\nu,2}(s))\big)ds\nonumber\\
=&&\!\!\!\!\!\!\!\!\int_0^t\big(B(s,Y^{h}_{\nu,1}(s))-B(s,Y^{h}_{\nu,2}(s))\big)\circ h(s)ds.
\end{eqnarray}
Applying the chain rule to $\|Y^{h}_{\nu,1}(t)-Y^{h}_{\nu,2}(t)\|^2_{F^*_{1,2,\nu}}$, similarly to \eref{eqn1}, we know
\begin{eqnarray*}
&&\!\!\!\!\!\!\!\!\|Y^{h}_{\nu,1}(t)-Y^{h}_{\nu,2}(t)\|^2_{F^*_{1,2,\nu}}+2\int_0^t\big\langle\Psi(Y^{h}_{\nu,1}(s))-\Psi(Y^{h}_{\nu,1}(s)),Y^{h}_{\nu,1}(s)-Y^{h}_{\nu,2}(s)\big\rangle_2ds\nonumber\\
=&&\!\!\!\!\!\!\!\!2\int_0^t\Big\langle \big(B(s,Y^{h}_{\nu,1}(s))-B(s,Y^{h}_{\nu,2}(s))\big)\circ h(s), Y^{h}_{\nu,1}(s)-Y^{h}_{\nu,2}(s)\Big\rangle_{F^*_{1,2,\nu}}ds.
\end{eqnarray*}
By \eref{eqn1.1} and \textbf{(H2)(i)}, we have
\begin{eqnarray*}
&&\!\!\!\!\!\!\!\!\|Y^{h}_{\nu,1}(t)-Y^{h}_{\nu,2}(t)\|^2_{F^*_{1,2,\nu}}+2\tilde{\alpha}\int_0^t|\Psi(Y^{h}_{\nu,1}(s))-\Psi(Y^{h}_{\nu,2}(s))|^2_2ds\nonumber\\
\leq&&\!\!\!\!\!\!\!\! \int_0^t2C_1|h(s)|_2\cdot \|Y^{h}_{\nu,1}(s)-Y^{h}_{\nu,2}(s)\|^2_{F^*_{1,2,\nu}}ds.
\end{eqnarray*}
Since $h\in\mathcal{S}_M$, by Gronwall's lemma, we get $Y^{h}_{\nu,1}=Y^{h}_{\nu,2}$, which implies the uniqueness. This completes the proof of Lemma 4.2.\hspace{\fill}$\Box$
\end{proof}

Now let's continue the proof of Theorem 4.1. The idea is to prove that the sequence $\{Y^{h}_{\nu}\}_{\nu\in(0,1)}$ converges to the solution of \eref{eq:29} as $\nu\rightarrow0$.

\vspace{2mm}
\begin{proof}
Rewrite \eref{eq:30} as
\begin{eqnarray*}
dY^{h}_\nu(t)+(1-L)\Psi(Y^{h}_\nu(t))dt=(1-\nu)\Psi(Y^{h}_\nu(t))dt+B(t,Y^{h}_\nu(t))\circ h(t)dt.
\end{eqnarray*}
Applying the chain rule to  $\frac{1}{2}\|Y^{h}_\nu(t)\|^2_{F^*_{1,2}}$, by \eref{eqn6} and \eref{eqn7}, we know that
\begin{eqnarray*}
&&\!\!\!\!\!\!\!\!\frac{1}{2}\|Y^{h}_\nu(t)\|^2_{F^*_{1,2}}+\int_0^t\big\langle\Psi(Y^{h}_\nu(s)),Y^{h}_\nu(s)\big\rangle_2ds\nonumber\\
=&&\!\!\!\!\!\!\!\!\frac{1}{2}\|x\|^2_{F^*_{1,2}}+(1-\nu)\int_0^t\big\langle\Psi(Y^{h}_\nu(s)),Y^{h}_\nu(s)\big\rangle_{F^*_{1,2}}ds\nonumber\\
&&\!\!\!\!\!\!\!\!+\int_0^t\big\langle B(s,Y^{h}_\nu(s))\circ h(s),Y^{h}_\nu(s)\big\rangle_{F^*_{1,2}}ds.
\end{eqnarray*}
Since $\Psi(r)r\geq\tilde{\alpha}|\Psi(r)|^2$, $\forall~r\in\Bbb{R}$, $L^2(\mu)\subset F^*_{1,2}$ continuously and densely, by \textbf{(H2)(ii)} and Young's inequality, we get
\begin{eqnarray*}
&&\!\!\!\!\!\!\!\!\frac{1}{2}\|Y^{h}_\nu(t)\|^2_{F^*_{1,2}}+\tilde{\alpha}\int_0^t|\Psi(Y^{h}_\nu(s))|_2^2ds\nonumber\\
\leq&&\!\!\!\!\!\!\!\!\frac{1}{2}\|x\|^2_{F^*_{1,2}}+\int_0^t\|\Psi(Y^{h}_\nu(s))\|_{F^*_{1,2}}\cdot\|Y^{h}_\nu(s)\|_{F^*_{1,2}}ds\nonumber\\
&&\!\!\!\!\!\!\!\!+\int_0^tC_1(\|Y^{h}_\nu(s)\|_{F^*_{1,2}}+1)\cdot|h(s)|_2\cdot\|Y^{h}_\nu(s)\|_{F^*_{1,2}}ds\nonumber\\
\leq&&\!\!\!\!\!\!\!\!\frac{1}{2}\|x\|^2_{F^*_{1,2}}+\frac{\tilde{\alpha}}{2}\int_0^t|\Psi(Y^{h}_\nu(s))|_2^2ds+\frac{1}{2\tilde{\alpha}}\int_0^t\|Y^{h}_\nu(s)\|^2_{F^*_{1,2}}ds\nonumber\\
&&\!\!\!\!\!\!\!\!+C_1\int_0^t(\|Y^{h}_\nu(s)\|_{F^*_{1,2}}+1)^2\cdot|h(s)|_2^2ds+\frac{1}{4C_1}\int_0^t\|Y^{h}_\nu(s)\|^2_{F^*_{1,2}}ds\nonumber\\
\leq&&\!\!\!\!\!\!\!\!\frac{1}{2}\|x\|^2_{F^*_{1,2}}+\frac{\tilde{\alpha}}{2}\int_0^t|\Psi(Y^{h}_\nu(s))|_2^2ds+\int_0^t(\frac{1}{2\tilde{\alpha}}+2C_1|h(s)|_2^2+\frac{1}{4C_1})\|Y^{h}_\nu(s)\|^2_{F^*_{1,2}}ds\nonumber\\
&&\!\!\!\!\!\!\!\!+2C_1\int_0^t|h(s)|_2^2ds,
\end{eqnarray*}
so by Gronwall's lemma and \eref{eq:32}, we know that
\begin{eqnarray}\label{eqn4}
&&\!\!\!\!\!\!\!\!\|Y^{h}_\nu(t)\|^2_{F^*_{1,2}}+\tilde{\alpha}\int_0^t|\Psi(Y^{h}_\nu(s))|_2^2ds\nonumber\\
\leq&&\!\!\!\!\!\!\!\! \Big\{\|x\|^2_{F^*_{1,2}}+4C_1M\Big\}\cdot\exp\Big\{\int_0^T\big(\frac{1}{2\tilde{\alpha}}+2C_1|h(s)|_2^2+\frac{1}{4C_1}\big)ds\Big\} \nonumber\\
:=&&\!\!\!\!\!\!\!\!C_{T,M}(\|x\|^2_{F^*_{1,2}}+1).
\end{eqnarray}
In the following, we will prove the convergence of $\{Y^{h}_\nu\}_{\nu\in(0,1)}$. Applying the chain rule to $\|Y^{h}_\nu(t)-Y^{h}_{\nu'}(t)\|^2_{F^*_{1,2}}$, $\nu,\nu'\in(0,1)$, we have that for all $t\in[0,T]$,
\begin{eqnarray}\label{eq:34}
&&\!\!\!\!\!\!\!\!\|Y^{h}_\nu(t)-Y^{h}_{\nu'}(t)\|^2_{F^*_{1,2}}+2\int_0^t\big\langle\Psi(Y^{h}_\nu(s))-\Psi(Y^{h}_{\nu'}(s)),Y^{h}_\nu(s)-Y^{h}_{\nu'}(s)\big\rangle_2ds\nonumber\\
=&&\!\!\!\!\!\!\!\!2\int_0^t\big\langle\Psi(Y^{h}_\nu(s))-\Psi(Y^{h}_{\nu'}(s)),Y^{h}_\nu(s)-Y^{h}_{\nu'}(s)\big\rangle_{F^*_{1,2}}ds\nonumber\\
&&\!\!\!\!\!\!\!\!-2\int_0^t\big\langle\nu\Psi(Y^{h}_\nu(s))-\nu'\Psi(Y^{h}_{\nu'}(s)),Y^{h}_\nu(s)-Y^{h}_{\nu'}(s)\big\rangle_{F^*_{1,2}}ds\nonumber\\
&&\!\!\!\!\!\!\!\!+2\int_0^t\big\langle\big(B(s,Y^{h}_\nu(s))-B(s,Y^{h}_{\nu'}(s))\big)\circ h(s),Y^{h}_\nu(s)-Y^{h}_{\nu'}(s)\rangle_{F^*_{1,2}}ds.
\end{eqnarray}
For the first term in the right hand-side of \eref{eq:34}, since $L^2(\mu)$ densely embedding into $F^*_{1,2}$, we have
\begin{eqnarray}\label{eqn10.1}
&&\!\!\!\!\!\!\!\!2\int_0^t\langle \Psi(Y^h_\nu(s))-\Psi(Y^h_{\nu'}(s)),Y^h_\nu(s)-Y^h_{\nu'}(s)\rangle_{F^*_{1,2}}ds\nonumber\\
\leq&&\!\!\!\!\!\!\!\!2\int_0^t|\Psi(Y^h_\nu(s))-\Psi(Y^h_{\nu'}(s))|_2\cdot\|Y^h_\nu(s)-Y^h_{\nu'}(s)\|_{F^*_{1,2}}ds.
\end{eqnarray}
Similarly, the second term in the right hand-side of \eref{eq:34} can be dominated by
\begin{eqnarray}\label{eqn10}
2\int_0^t\big(\nu|\Psi(Y^{h}_\nu(s))|_2+\nu'|\Psi(Y^{h}_{\nu'}(s))|_2\big)\cdot\|Y^{h}_\nu(s)-Y^{h}_{\nu'}(s)\|_{F^*_{1,2}}ds.
\end{eqnarray}
By \textbf{(H2)(i)}, the third term in the right hand-side of \eref{eq:34} can be dominated by
\begin{eqnarray}\label{eqn10.2}
2C_1\int_0^t\|Y^{h}_{\nu}(s)-Y^{h}_{\nu'}(s)\|_{F^*_{1,2}}\cdot|h(s)|_2\cdot\|Y^{h}_{\nu}(s)-Y^{h}_{\nu'}(s)\|_{F^*_{1,2}}ds.
\end{eqnarray}
Plugging \eref{eqn10.1}-\eref{eqn10.2} into \eref{eq:34}, by \eref{eqn1.1} and Young's inequality, we get
\begin{eqnarray}\label{eq:35}
&&\!\!\!\!\!\!\!\!\|Y^{h}_\nu(t)-Y^{h}_{\nu'}(t)\|^2_{F^*_{1,2}}+2\tilde{\alpha}\int_0^t|\Psi(Y^{h}_{\nu'}(s))-\Psi(Y^{h}_{\nu}(s))|_2^2ds\nonumber\\
\leq&&\!\!\!\!\!\!\!\!2\int_0^t|\Psi(Y^{h}_{\nu}(s))-\Psi(Y^{h}_{\nu'}(s))|_2\cdot\|Y^{h}_{\nu}(s)-Y^{h}_{\nu'}(s)\|_{F^*_{1,2}}ds\nonumber\\
&&\!\!\!\!\!\!\!\!+2\int_0^t\big(\nu|\Psi(Y^{h}_{\nu}(s))|_2+\nu'|\Psi(Y^{h}_{\nu'}(s))|_2\big)\cdot\|Y^{h}_{\nu}(s)-Y^{h}_{\nu'}(s)\|_{F^*_{1,2}}ds\nonumber\\
&&\!\!\!\!\!\!\!\!+2C_1\int_0^t\|Y^{h}_{\nu}(s)-Y^{h}_{\nu'}(s)\|_{F^*_{1,2}}\cdot|h(s)|_2\cdot\|Y^{h}_{\nu}(s)-Y^{h}_{\nu'}(s)\|_{F^*_{1,2}}ds\nonumber\\
\leq&&\!\!\!\!\!\!\!\!\tilde{\alpha}\int_0^t|\Psi(Y^{h}_{\nu'}(s))-\Psi(Y^{h}_{\nu}(s))|_2^2ds+\frac{1}{\tilde{\alpha}}\int_0^t\|Y^{h}_{\nu}(s)-Y^{h}_{\nu'}(s)\|^2_{F^*_{1,2}}ds\nonumber\\
&&\!\!\!\!\!\!\!\!+\int_0^t\|Y^{h}_{\nu}(s)-Y^{h}_{\nu'}(s)\|^2_{F^*_{1,2}}ds+2\int_0^t\nu^2|\Psi(Y^{h}_{\nu}(s))|_2^2+\nu'^2|\Psi(Y^{h}_{\nu'}(s))|_2^2ds\nonumber\\
&&\!\!\!\!\!\!\!\!+\int_0^t\|Y^{h}_{\nu}(s)-Y^{h}_{\nu'}(s)\|^2_{F^*_{1,2}}\cdot|h(s)|_2^2ds+C_1^2\int_0^t\|Y^{h}_{\nu}(s)-Y^{h}_{\nu'}(s)\|^2_{F^*_{1,2}}ds,
\end{eqnarray}
which yields
\begin{eqnarray*}
&&\!\!\!\!\!\!\!\!\sup_{s\in[0,T]}\|Y^{h}_{\nu}(s)-Y^{h}_{\nu'}(s)\|^2_{F^*_{1,2}}+\tilde{\alpha}\int_0^t|\Psi(Y^{h}_{\nu}(s))-\Psi(Y^{h}_{\nu'}(s))|_2^2ds\nonumber\\
\leq&&\!\!\!\!\!\!\!\!\int_0^t(\frac{1}{\tilde{\alpha}}+1+|h(s)|_2^2+C_1^2)\|Y^{h}_{\nu}(s)-Y^{h}_{\nu'}(s)\|^2_{F^*_{1,2}}ds\nonumber\\
&&\!\!\!\!\!\!\!\!+2(\nu^2+\nu'^2)\int_0^t|\Psi(Y^{h}_{\nu}(s))|_2^2+|\Psi(Y^{h}_{\nu'}(s))|_2^2ds.
\end{eqnarray*}
Recall that if $x\in L^2(\mu)$, then we have \eref{eq:32}, if $x\in F^*_{1,2}$ and \eref{eq:2.2} is satisfied, then we have \eref{eqn4}. Hence, by Gronwall's inequality, we know that there exists a positive constant $C_{T,M}\in(0,\infty)$ which is independent of $\nu,~ \nu'$ such that
\begin{eqnarray}\label{eq:36}
&&\!\!\!\!\!\!\!\!\sup_{s\in[0,T]}\|Y^{h}_{\nu}(s)-Y^{h}_{\nu'}(s)\|^2_{F^*_{1,2}}+\tilde{\alpha}\int_0^T|\Psi(Y^{h}_{\nu}(s))-\Psi(Y^{h}_{\nu'}(s))|_2^2ds\nonumber\\
\leq&&\!\!\!\!\!\!\!\!C_{T,M}(\nu^2+\nu'^2).
\end{eqnarray}
Hence, there exists a function $Y^h\in C([0,T];F^*_{1,2})$ such that $Y^{h}_\nu\rightarrow Y^{h}$ in $C([0,T];F^*_{1,2})$ as $\nu\rightarrow0$.

Next we will prove $Y^{h}$ satisfies \eref{eq:29}.
By
\begin{eqnarray*}\label{eq:37}
Y^{h}_\nu\rightarrow Y^{h}~~\text{in}~~C([0,T];F^*_{1,2}),~~\text{as}~~\nu\rightarrow0,
\end{eqnarray*}
use an argument similar to that used in Claim 4.3, we have
\begin{eqnarray*}
\int_0^\cdot B(s,Y^{h}_\nu(s))\circ h(s)ds\rightarrow \int_0^\cdot B(s,Y^{h}(s))\circ h(s)ds~~\text{in}~~C([0,T];F^*_{1,2}),~~\text{as}~~\nu\rightarrow0
\end{eqnarray*}
and
\begin{eqnarray*}
\int_0^\cdot\Psi(Y^{h}_\nu(s))ds\rightarrow \int_0^\cdot\Psi(Y^{h}(s))ds~~\text{in}~~L^2([0,T];L^2(\mu)),~~\text{as}~~\nu\rightarrow0.
\end{eqnarray*}
Hence $Y^{h}$ satisfies \eref{eq:29}. This completes the proof of the existence in Theorem 4.1.
\vspace{2mm}

\textbf{Uniqueness}
\vspace{2mm}

Assume $Y^{h}_1$ and $Y^{h}_2$ are two solutions to \eref{eq:29}, we know that
\begin{eqnarray}\label{eq:38}
&&\!\!\!\!\!\!\!\!Y^{h}_1(t)-Y^{h}_2(t)-L\int_0^t \Psi(Y^{h}_1(s))-\Psi(Y^{h}_2(s))ds\nonumber\\
=&&\!\!\!\!\!\!\!\!\int_0^t\big(B(s,Y^{h}_1(s))-B(s,Y^{h}_2(s))\big)\circ h(s)ds,~~\forall~t\in[0,T].
\end{eqnarray}
Rewrite \eref{eq:38} as
\begin{eqnarray}\label{eq:39}
&&\!\!\!\!\!\!\!\!Y^{h}_1(t)-Y^{h}_2(t)+(1-L)\int_0^t \Psi(Y^{h}_1(s))-\Psi(Y^{h}_2(s))ds\nonumber\\
=&&\!\!\!\!\!\!\!\!\int_0^t\Psi(Y^{h}_1(s))-\Psi(Y^{h}_2(s))ds\nonumber\\
&&\!\!\!\!\!\!\!\!+\int_0^t\big(B(s,Y^{h}_1(s))-B(s,Y^{h}_2(s))\big)\circ h(s)ds,~\forall~ t\in[0,T].
\end{eqnarray}
Applying the chain rule to $\|Y^{h}_1(t)-Y^{h}_2(t)\|^2_{F^*_{1,2}}$, by \eref{eqn6} and \eref{eqn7}, we have that
\begin{eqnarray}\label{eq:40}
&&\!\!\!\!\!\!\!\!\|Y^{h}_1(t)-Y^{h}_2(t)\|^2_{F^*_{1,2}}+2\int_0^t\big\langle\Psi(Y^{h}_1(s))-\Psi(Y^{h}_2(s))\big\rangle_2ds\nonumber\\
=&&\!\!\!\!\!\!\!\!2\int_0^t\big\langle\Psi(Y^{h}_1(s))-\Psi(Y^{h}_2(s)),Y^{h}_1(s)-Y^{h}_2(s)\big\rangle_{F^*_{1,2}}ds\nonumber\\
&&\!\!\!\!\!\!\!\!+2\int_0^t\big\langle\big(B(s,Y^{h}_1(s))-B(s,Y^{h}_2(s))\big)\circ h(s),Y^{h}_1(s)-Y^{h}_2(s)\big\rangle_{F^*_{1,2}}ds.
\end{eqnarray}
By \eref{eqn1.1} and \textbf{(H2)(i)}, we get
\begin{eqnarray*}
&&\!\!\!\!\!\!\!\!\|Y^{h}_1(t)-Y^{h}_2(t)\|^2_{F^*_{1,2}}+2\tilde{\alpha}\int_0^t|\Psi(Y^{h}_1(s))-\Psi(Y^{h}_2(s))|_2^2ds\nonumber\\
\leq&&\!\!\!\!\!\!\!\!2\int_0^t\|\Psi(Y^{h}_1(s))-\Psi(Y^{h}_2(s))\|_{F^*_{1,2}}\cdot\|Y^{h}_1(s)-Y^{h}_2(s)\|_{F^*_{1,2}}ds\nonumber\\
&&\!\!\!\!\!\!\!\!+2C_1\int_0^t\|Y^{h}_1(s)-Y^{h}_2(s)\|_{F^*_{1,2}}\cdot|h(s)|_2\cdot\|Y^{h}_1(s)-Y^{h}_2(s)\|_{F^*_{1,2}}ds\textcolor{blue}{.}
\end{eqnarray*}
Since $L^2(\mu)\subset F^*_{1,2}$ continuously and densely, using Young's inequality, we obtain that
\begin{eqnarray*}
&&\!\!\!\!\!\!\!\!\|Y^{h}_1(t)-Y^{h}_2(t)\|^2_{F^*_{1,2}}+2\tilde{\alpha}\int_0^t|\Psi(Y^{h}_1(s))-\Psi(Y^{h}_2(s))|_2^2ds\nonumber\\
\leq&&\!\!\!\!\!\!\!\!2\tilde{\alpha}\int_0^t|\Psi(Y^{h}_1(s))-\Psi(Y^{h}_2(s))|_2^2ds+\big(\frac{1}{2\tilde{\alpha}}+1\big)\int_0^t\|Y^{h}_1(s)-Y^{h}_2(s)\|^2_{F^*_{1,2}}ds\nonumber\\
&&\!\!\!\!\!\!\!\!+C_1^2\int_0^t|h(s)|_2^2\cdot\|Y^{h}_1(s)-Y^{h}_2(s)\|^2_{F^*_{1,2}}ds.
\end{eqnarray*}
Therefore,
$$\|Y^{h}_1(t)-Y^{h}_2(t)\|^2_{F^*_{1,2}}\leq\int_0^t(\frac{1}{2\tilde{\alpha}}+1+C_1^2|h(s)|_2^2)\cdot\|Y^{h}_1(s)-Y^{h}_2(s)\|^2_{F^*_{1,2}}ds.$$
Since $h\in\mathcal{S}_M$, by Gronwall's lemma, we get $Y^{h}_1=Y^{h}_2$. This completes the proof of Theorem 4.1. \hspace{\fill}$\Box$
\end{proof}

\section{Large deviations}
\setcounter{equation}{0}
 \setcounter{definition}{0}

This section is devoted to check \textbf{(a)} and \textbf{(b)} in the proof of Theorem \ref{Th 3.3}.
The verification of \textbf{(a)} will be given in Theorem \ref{Th 5.1}. \textbf{(b)} will
be established in Theorem \ref{Th 5.2}. Assuming these have been done, the
proof of Theorem \ref{Th 3.3} is complete.

\begin{theorem}\label{Th 5.1}
For every $M<\infty$, let $\{h_\varepsilon:\varepsilon>0\}\subset \mathcal{A}_M$. Then
$$\lim_{\varepsilon\rightarrow0}\mathbb{E}\Big\|\mathcal{G}^\varepsilon\big(W(\cdot)+\frac{1}{\sqrt{\varepsilon}}\int_0^\cdot h_\varepsilon(s)ds\big)-\mathcal{G}^0\big(\int_0^\cdot h_\varepsilon(s)ds\big)\Big\|_{C([0,T];F^*_{1,2})}^2=0.$$
\end{theorem}
\begin{proof}
Recall the definition of $\mathcal{G}^0$; see (\ref{eq. g0}). We know that $\mathcal{G}^0\big(\int_0^\cdot h_\varepsilon(s)ds\big)$ is the strong solution to \eref{eq:3}:
\begin{equation}\label{eq:3}
\left\{
  \begin{array}{ll}
    dY^{h_\varepsilon}(t)-L\Psi(Y^{h_\varepsilon}(t))dt=B(t,Y^{h_\varepsilon}(t))\circ h_\varepsilon(t)dt, \\
    Y^{h_\varepsilon}(0)=x\in L^2(\mu).
  \end{array}
\right.
\end{equation}
Recall the definition of $\mathcal{G}^\varepsilon$; see (\ref{eq. g1}). According to Yamada-Watanabe theorem (cf.\cite{RSZ}) and Theorem \ref{Th 3.1}, $\mathcal{G}^\varepsilon\big(W(\cdot)+\frac{1}{\sqrt{\varepsilon}}\int_0^\cdot h_\varepsilon(s)ds\big)$
is the strong solution to \eref{eq:2}:
%Let $\{X^{h_\varepsilon}\}_{\varepsilon>0}$ be a family of solutions to the following equation
\begin{equation}\label{eq:2}
 \left\{
   \begin{array}{ll}
     dX^{h_\varepsilon}(t)-L\Psi(X^{h_\varepsilon}(t))dt=B(t,X^{h_\varepsilon}(t))\circ h_\varepsilon(t)dt+\sqrt{\varepsilon}B(t,X^{h_\varepsilon}(t))dW(t), \\
     X^{h_\varepsilon}(0)=x\in L^2(\mu).
   \end{array}
 \right.
\end{equation}
For simplicity, we denote
\begin{eqnarray}\label{eq:5.1}
X^{h_\varepsilon}:=\mathcal{G}^\varepsilon\big(W(\cdot)+\frac{1}{\sqrt{\varepsilon}}\int_0^\cdot h_\varepsilon(s)ds\big) ~~\text{and}~~Y^{h_\varepsilon}:=\mathcal{G}^0\big(\int_0^\cdot h_\varepsilon(s)ds\big),
\end{eqnarray}
in addition, we have
\begin{eqnarray*}\label{eq:4}
X^{h_\varepsilon}(t)-Y^{h_\varepsilon}(t)=&&\!\!\!\!\!\!\!\!L\int_0^t(\Psi(X^{h_\varepsilon}(s))-\Psi(Y^{h_\varepsilon}(s)))ds\nonumber\\
&&\!\!\!\!\!\!\!\!+\int_0^t\big(B(s,X^{h_\varepsilon}(s))-B(s,Y^{h_\varepsilon}(s))\big)\circ h_\varepsilon(s)ds\nonumber\\
&&\!\!\!\!\!\!\!\!+\int_0^t\sqrt{\varepsilon} B(s,X^{h_\varepsilon}(s))dW(s).
\end{eqnarray*}
Rewrite the above equation in the following form
\begin{eqnarray*}\label{eq:5}
  &&\!\!\!\!\!\!\!\!\big(X^{h_\varepsilon}(t)-Y^{h_\varepsilon}(t)\big)+(1-L)\int_0^t\big(\Psi(X^{h_\varepsilon}(s))-\Psi(Y^{h_\varepsilon}(s))\big)ds\nonumber\\
=&&\!\!\!\!\!\!\!\!\int_0^t(\Psi(X^{h_\varepsilon}(s))-\Psi(Y^{h_\varepsilon}(s)))ds+\int_0^t\big(B(s,X^{h_\varepsilon}(s))-B(s,Y^{h_\varepsilon}(s))\big)\circ h_\varepsilon(s)ds\nonumber\\
  &&\!\!\!\!\!\!\!\!+\int_0^t\sqrt{\varepsilon} B(s,X^{h_\varepsilon}(s))dW(s).
\end{eqnarray*}
Applying the It\^{o} formula (\cite[Theorem 1.1]{LR1}) to $\|X^{h_\varepsilon}(t)-Y^{h_\varepsilon}(t)\|^2_{F^*_{1,2}}$, by \eref{eqn7}, we obtain
\begin{eqnarray*}\label{eq:6}
&&\!\!\!\!\!\!\!\!\|X^{h_\varepsilon}(t)-Y^{h_\varepsilon}(t)\|^2_{F^*_{1,2}}+2\int_0^t\big\langle\Psi(X^{h_\varepsilon}(s))-\Psi(Y^{h_\varepsilon}(s)),X^{h_\varepsilon}(s)-Y^{h_\varepsilon}(s)\big\rangle_2ds\nonumber\\
=&&\!\!\!\!\!\!\!\!2\int_0^t\big\langle\Psi(X^{h_\varepsilon}(s))-\Psi(Y^{h_\varepsilon}(s)),X^{h_\varepsilon}(s)-Y^{h_\varepsilon}(s)\big\rangle_{F^*_{1,2}}ds\nonumber\\
&&\!\!\!\!\!\!\!\!+2\int_0^t\big\langle\big(B(s,X^{h_\varepsilon}(s))-B(s,Y^{h_\varepsilon}(s))\big)\circ h_\varepsilon(s),X^{h_\varepsilon}(s)-Y^{h_\varepsilon}(s)\big\rangle_{F^*_{1,2}}ds\nonumber\\
&&\!\!\!\!\!\!\!\!+2\int_0^t\big\langle X^{h_\varepsilon}(s)-Y^{h_\varepsilon}(s),\sqrt{\varepsilon} B(s,X^{h_\varepsilon}(s))dW(s)\big\rangle_{F^*_{1,2}}\nonumber\\
&&\!\!\!\!\!\!\!\!+\int_0^t\|\sqrt{\varepsilon}B(s,X^{h_\varepsilon}(s))\|^2_{L_2(L^2(\mu),F^*_{1,2})}ds.
\end{eqnarray*}
Since $L^2(\mu)\subset F^*_{1,2}$ continuously and densely, by \eref{eqn1.1}, Young's inequality and \textbf{(H2)}, we get
\begin{eqnarray*}\label{eq:7}
&&\!\!\!\!\!\!\!\!\|X^{h_\varepsilon}(t)-Y^{h_\varepsilon}(t)\|^2_{F^*_{1,2}}+2\widetilde{\alpha}\int_0^t|\Psi(X^{h_\varepsilon}(s))-\Psi(Y^{h_\varepsilon}(s))|^2_2ds\nonumber\\
\leq&&\!\!\!\!\!\!\!\!2\widetilde{\alpha}\int_0^t|\Psi(X^{h_\varepsilon}(s))-\Psi(Y^{h_\varepsilon}(s))|_2^2ds+\frac{1}{2\widetilde{\alpha}}\int_0^t\|X^{h_\varepsilon}(s)-Y^{h_\varepsilon}(s)\|^2_{F^*_{1,2}}ds\nonumber\\
&&\!\!\!\!\!\!\!\!+\int_0^tC_1\|X^{h_\varepsilon}(s)-Y^{h_\varepsilon}(s)\|^2_{F^*_{1,2}}\cdot |h_\varepsilon(s)|^2_2ds\nonumber\\
&&\!\!\!\!\!\!\!\!+\int_0^t\|X^{h_\varepsilon}(s)-Y^{h_\varepsilon}(s)\|^2_{F^*_{1,2}}ds+2\Big|\int_0^t\big\langle X^{h_\varepsilon}(s)-Y^{h_\varepsilon}(s),\sqrt{\varepsilon} B(s,X^{h_\varepsilon}(s))dW(s)\big\rangle_{F^*_{1,2}}\Big|\nonumber\\
&&\!\!\!\!\!\!\!\!+2C_2\varepsilon\int_0^t(\|X^{h_\varepsilon}(s)\|^2_{F^*_{1,2}}+1)ds,
\end{eqnarray*}
this yields,
\begin{eqnarray*}\label{eq:8}
&&\!\!\!\!\!\!\!\!\|X^{h_\varepsilon}(t)-Y^{h_\varepsilon}(t)\|^2_{F^*_{1,2}}\nonumber\\
\leq&&\!\!\!\!\!\!\!\!\int_0^t(\frac{1}{2\widetilde{\alpha}}+C_1|h_\varepsilon(s)|^2_2+1)\cdot \|X^{h_\varepsilon}(s)-Y^{h_\varepsilon}(s)\|^2_{F^*_{1,2}}ds\nonumber\\
&&\!\!\!\!\!\!\!\!+2\Big|\int_0^t\big\langle X^{h_\varepsilon}(s)-Y^{h_\varepsilon}(s),\sqrt{\varepsilon} B(s, X^{h_\varepsilon}(s))dW(s)\big\rangle_{F^*_{1,2}}\Big|\nonumber\\
&&\!\!\!\!\!\!\!\!+2\varepsilon C_2\int_0^t(\|X^{h_\varepsilon}(s)\|^2_{F^*_{1,2}}+1)ds.
\end{eqnarray*}
By Gronwall's lemma, we get
\begin{eqnarray*}\label{eq:9}
&&\!\!\!\!\!\!\!\!\|X^{h_\varepsilon}(t)-Y^{h_\varepsilon}(t)\|^2_{F^*_{1,2}}\nonumber\\
\leq&&\!\!\!\!\!\!\!\!\Bigg[2\sup_{t\in [0,T]}\Big|\int_0^t\big\langle X^{h_\varepsilon}(s)-Y^{h_\varepsilon}(s),\sqrt{\varepsilon} B(s,X^{h_\varepsilon}(s))dW(s) \big\rangle_{F^*_{1,2}}\Big|\nonumber\\
&&\!\!\!\!\!\!\!\!+2\varepsilon C_2\int_0^T(\|X^{h_\varepsilon(s)}\|^2_{F^*_{1,2}}+1)ds\Bigg]\cdot \exp\Big\{\int_0^T\big(\frac{1}{2\widetilde{\alpha}}+C_1|h_\varepsilon(s)|^2_2+1\big)ds\Big\}.
\end{eqnarray*}
Since $\int_0^T|h_\varepsilon(s)|^2_2ds\leq M$, $\Bbb{P}$-a.s., then by BDG's inequality and \textbf{(H2)(ii)}, we know
\begin{eqnarray*}\label{eq:10}
&&\!\!\!\!\!\!\!\!\Bbb{E}\Big[\sup_{t\in[0,T]}\|X^{h_\varepsilon}(t)-Y^{h_\varepsilon}(t)\|^2_{F^*_{1,2}}\Big]\nonumber\\
\leq&&\!\!\!\!\!\!\!\! C_{T,M}\Bigg[\Bbb{E}\sup_{t\in[0,T]}\Big|\int_0^t\big\langle X^{h_\varepsilon}(s)-Y^{h_\varepsilon}(s),\sqrt{\varepsilon} B(s,X^{h_\varepsilon}(s))dW(s)\big\rangle_{F^*_{1,2}}\Big|+\varepsilon\Bbb{E}\int_0^T(\|X^{h_\varepsilon}(s)\|^2_{F^*_{1,2}}+1)ds\Bigg]\nonumber\\
\leq&&\!\!\!\!\!\!\!\! C_{T,M}\Bbb{E}\Bigg[\int_0^T\|X^{h_\varepsilon}(s)-Y^{h_\varepsilon}(s)\|^2_{F^*_{1,2}}\cdot\|\sqrt{\varepsilon} B(s,X^{h_\varepsilon(s)})\|^2_{L_2(L^2(\mu),F^*_{1,2}))}ds\Bigg]^{\frac{1}{2}}\nonumber\\
&&\!\!\!\!\!\!\!\! +C_{T,M}\Bbb{E}\Big[\varepsilon\int_0^T(\|X^{h_\varepsilon}(s)\|^2_{F^*_{1,2}}+1)ds\Big]\nonumber\\
\leq&&\!\!\!\!\!\!\!\!C_{T,M}\Bbb{E}\Bigg[\sup_{t\in[0,T]}\|X^{h_\varepsilon}(t)-Y^{h_\varepsilon}(t)\|^2_{F^*_{1,2}}\cdot 2C_2\varepsilon\int_0^t(\|X^{h_\varepsilon}(s)\|^2_{F^*_{1,2}}+1)ds\Bigg]^{\frac{1}{2}}\nonumber\\
&&\!\!\!\!\!\!\!\!+C_{T,M}\Bbb{E}\Big[\varepsilon\int_0^T(\|X^{h_\varepsilon}(s)\|^2_{F^*_{1,2}}+1)ds\Big]\nonumber\\
=&&\!\!\!\!\!\!\!\!C_{T,M}\Bbb{E}\Bigg[\sup_{t\in[0,T]}\|X^{h_\varepsilon}(t)-Y^{h_\varepsilon}(t)\|_{F^*_{1,2}}\cdot\Big(2C_2\varepsilon\int_0^t(\|X^{h_\varepsilon}(s)\|^2_{F^*_{1,2}}+1)ds\Big)^{\frac{1}{2}}\Bigg]\nonumber\\
&&\!\!\!\!\!\!\!\!+C_{T,M}\Bbb{E}\Big[\varepsilon\int_0^T(\|X^{h_\varepsilon}(s)\|^2_{F^*_{1,2}}+1)ds\Big]\nonumber\\
\leq&&\!\!\!\!\!\!\!\! \frac{1}{2}\Bbb{E}\Bigg[\sup_{t\in[0,T]}\|X^{h_\varepsilon}(t)-Y^{h_\varepsilon}(t)\|^2_{F^*_{1,2}}\Bigg]+\varepsilon C_{T,M}\Bbb{E}\Bigg[\int_0^T(\|X^{h_\varepsilon}(s)\|^2_{F^*_{1,2}}+1)ds\Bigg],
\end{eqnarray*}
which yields
\begin{eqnarray}\label{eq:11}
\Bbb{E}\Bigg[\sup_{t\in[0,T]}\|X^{h_\varepsilon}(t)-Y^{h_\varepsilon}(t)\|^2_{F^*_{1,2}}\Bigg]\leq \varepsilon C_{T,M}\Bbb{E}\Bigg[\int_0^T(\|X^{h_\varepsilon}(s)\|^2_{F^*_{1,2}}+1)ds\Bigg].
\end{eqnarray}
%To get the convergence of $X^{h_\varepsilon}\longrightarrow Y^{h_\varepsilon}$ in $L^2(\Omega,C([0,T];F^*_{1,2}))$ as $\varepsilon\rightarrow0$, we need the uniform boundedness of $\|X^{h_\varepsilon}(t)\|^2_{F^*_{1,2}}$, which can be obtained by
Applying the It\^o's formula to $\|X^{h_\varepsilon}(t)\|^2_{F^*_{1,2}}$, and using a similar argument as \eref{eq:11}, we can get, for any $\varepsilon\in(0,1)$,
\begin{eqnarray*}
\Bbb{E}\Bigg[\sup_{t\in[0,T]}\|X^{h_\varepsilon}(t)\|^2_{F^*_{1,2}}\Bigg]\leq C_{T,M}\Bbb{E}\Bigg[\int_0^T(\|X^{h_\varepsilon}(s)\|^2_{F^*_{1,2}}+1)ds\Bigg],
\end{eqnarray*}
which implies that
\begin{eqnarray*}
\sup_{\varepsilon\in(0,1)}\Bbb{E}\Bigg[\sup_{t\in[0,T]}\|X^{h_\varepsilon}(t)\|^2_{F^*_{1,2}}\Bigg]\leq C_{T,M}<\infty.
\end{eqnarray*}
Combining this with \eref{eq:11}, we get that $X^{h_\varepsilon}\longrightarrow Y^{h_\varepsilon}$ in $L^2(\Omega,C([0,T];F^*_{1,2}))$ as $\varepsilon\rightarrow0$, which implies \eref{eq:5.1} and consequently \textbf{(a)} holds.\hspace{\fill}$\Box$
\end{proof}

Let's continue to check \textbf{(b)}. Let $\{h, h_n\in\mathcal{S}_M, n\geq1\}$, by the definition of $\mathcal{G}^0$ (see (\ref{eq. g0})), we know that $X^h:=\mathcal{G}^0(\int_0^\cdot h(t)dt)$ is the unique strong solution to the following equation:
\begin{equation}\label{eq:4}
\left\{
  \begin{array}{ll}
    dX^{h}(t)-L\Psi(X^{h}(t))dt=B(t,X^{h}(t))\circ h(t)dt,\\
    X^{h}(0)=x\in L^2(\mu),
  \end{array}
\right.
\end{equation}
meanwhile, $X^{h_n}:=\mathcal{G}^0(\int_0^\cdot h_n(t)dt)$ is the unique strong solution to the following equation:
\begin{equation}\label{eq:4.1}
\left\{
  \begin{array}{ll}
    dX^{h_n}(t)-L\Psi(X^{h_n}(t))dt=B(t,X^{h_n}(t))\circ h_n(t)dt,\\
    X^{h_n}(0)=x\in L^2(\mu).
  \end{array}
\right.
\end{equation}

The proof of Theorem \ref{Th 5.2} is inspired by \cite{CM}.
\begin{theorem}\label{Th 5.2}
Assume that $h_n\rightarrow h$ weakly in $L^2([0,T];L^2(\mu))$ as $n\rightarrow\infty$. Then $\mathcal{G}^0(\int_0^\cdot h_n(t)dt)$ converges to $\mathcal{G}^0(\int_0^\cdot h(t)dt)$ in $C([0,T];F^*_{1,2})$ as $n\rightarrow\infty$.
\end{theorem}
\begin{proof}
Notice that
\begin{eqnarray*}\label{eq:13}
&&\!\!\!\!\!\!\!\!d(X^{h_n}(t)-X^h(t))-L\big(\Psi(X^{h_n}(t))-\Psi(X^h(t))\big)dt\nonumber\\
=&&\!\!\!\!\!\!\!\!\big(B(t,X^{h_n}(t))\circ h_n(t)-B(t,X^h(t))\circ h(t)\big)dt.
\end{eqnarray*}
Rewrite the above equation as
\begin{eqnarray*}\label{eq:14}
&&\!\!\!\!\!\!\!\!d\big(X^{h_n}(t)-X^h(t)\big)+(1-L)\big(\Psi(X^{h_n}(t))-\Psi(X^h(t))\big)dt\nonumber\\
=&&\!\!\!\!\!\!\!\!\big(\Psi(X^{h_n}(t))-\Psi(X^h(t))\big)dt+\big(B(t,X^{h_n}(t))\circ h_n(t)-B(t,X^h(t))\circ h(t)\big)dt.
\end{eqnarray*}
Applying the chain rule to $\|X^{h_n}-X^h\|^2_{F^*_{1,2}}$, by \eref{eqn7} we get
\begin{eqnarray*}\label{eq:15}
&&\!\!\!\!\!\!\!\!\|X^{h_n}(t)-X^h(t)\|^2_{F^*_{1,2}}+2\int_0^t\big\langle\Psi(X^{h_n}(s))-\Psi(X^h(s)),X^{h_n}(s)-X^h(s)\big\rangle_2ds\nonumber\\
=&&\!\!\!\!\!\!\!\!2\int_0^t\big\langle\Psi(X^{h_n}(s))-\Psi(X^h(s)),X^{h_n}(s)-X^h(s)\big\rangle_{F^*_{1,2}}ds\nonumber\\
&&\!\!\!\!\!\!\!\!+2\int_0^t\big\langle B(s,X^{h_n}(s))\circ h_n(s)-B(s,X^h(s))\circ h(s),X^{h_n}(s)-X^h(s)\big\rangle_{F^*_{1,2}}ds\nonumber\\
=&&\!\!\!\!\!\!\!\!2\int_0^t\big\langle\Psi(X^{h_n}(s))-\Psi(X^h(s)),X^{h_n}(s)-X^h(s)\big\rangle_{F^*_{1,2}}ds\nonumber\\
&&\!\!\!\!\!\!\!\!+2\int_0^t\big\langle B(s,X^{h_n}(s))\circ h_n(s)-B(s,X^h(s))\circ h_n(s),X^{h_n}(s)-X^h(s)\big\rangle_{F^*_{1,2}}ds\nonumber\\
&&\!\!\!\!\!\!\!\!+2\int_0^t\big\langle B(s,X^h(s))\circ h_n(s)-B(s,X^h(s))\circ h(s),X^{h_n}(s)-X^h(s)\big\rangle_{F^*_{1,2}}ds.
\end{eqnarray*}
By \eref{eqn1.1}, since $L^2(\mu)\subset F^*_{1,2}$ continuously and densely, using \textbf{(H2)(i)} and Young's inequality, we obtain
\begin{eqnarray*}\label{eq:16}
&&\!\!\!\!\!\!\!\!\|X^{h_n}(t)-X^h(t)\|^2_{F^*_{1,2}}+2\widetilde{\alpha}\int_0^t|\Psi(X^{h_n}(s))-\Psi(X^h(s))|_2^2ds\nonumber\\
\leq&&\!\!\!\!\!\!\!\! 2\int_0^t|\Psi(X^{h_n}(s))-\Psi(X^h(s))|_2\cdot\|X^{h_n}(s)-X^h(s)\|_{F^*_{1,2}}ds\nonumber\\
&&\!\!\!\!\!\!\!\!+2\int_0^t\|B(s,X^{h_n}(s))-B(s,X^h(s))\|_{L_2(L^2(\mu),F^*_{1,2})}\cdot|h_n(s)|_2\cdot\|X^{h_n}(s)-X^h(s)\|_{F^*_{1,2}}ds\nonumber\\
&&\!\!\!\!\!\!\!\!+2\int_0^t\big\langle B(s,X^h(s))\circ h_n(s)-B(s,X^h(s))\circ h(s),X^{h_n}(s)-X^h(s)\big\rangle_{F^*_{1,2}}ds\nonumber\\
\leq&&\!\!\!\!\!\!\!\! 2\widetilde{\alpha}\int_0^t|\Psi(X^{h_n}(s))-\Psi(X^h(s))|_2^2ds+\frac{1}{2\widetilde{\alpha}}\int_0^t\|X^{h_n}(s)-X^h(s)\|^2_{F^*_{1,2}}ds\nonumber\\
&&\!\!\!\!\!\!\!\! +2C_1\int_0^t\|X^{h_n}(s)-X^h(s)\|^2_{F^*_{1,2}}\cdot |h_n(s)|_2ds\nonumber\\
&&\!\!\!\!\!\!\!\!+2\int_0^t\big\langle B(s,X^h(s))\circ h_n(s)-B(s,X^h(s))\circ h(s),X^{h_n}(s)-X^h(s)\big\rangle_{F^*_{1,2}}ds.
\end{eqnarray*}

For simplicity, denote
$$
I_n(t):=2\int_0^t\big\langle B(s,X^h(s))\circ h_n(s)-B(s,X^h(s))\circ h(s),X^{h_n}(s)-X^h(s)\big\rangle_{F^*_{1,2}}ds.
$$
By Gronwall's inequality, we obtain
\begin{eqnarray}\label{eq:17}
&&\!\!\!\!\!\!\!\!\sup_{t\in[0,T]}\|X^{h_n}(t)-X^h(t)\|^2_{F^*_{1,2}}\nonumber\\
\leq&&\!\!\!\!\!\!\!\! \sup_{t\in[0,T]}|I_n(t)|\cdot\exp\big\{\frac{1}{2\widetilde{\alpha}}T+2C_1\int_0^T|h_n(s)|_2ds\big\}.
\end{eqnarray}
Since
$$
2\int_0^T|h_n(s)|_2ds\leq2\big(\int_0^T|h_n(s)|_2^2ds\big)^{\frac{1}{2}}\cdot\big(\int_0^T1ds\big)^{\frac{1}{2}}=2M^{\frac{1}{2}}\cdot T^{\frac{1}{2}},
$$
\eref{eq:17} yields
\begin{eqnarray}\label{eqn1.2}
&&\sup_{t\in[0,T]}\|X^{h_n}(t)-X^h(t)\|^2_{F^*_{1,2}}\leq C_{M,T}\cdot\sup_{t\in[0,T]}|I_n(t)|.
\end{eqnarray}
To estimate $|I_n(t)|$, we denote
$$
U^n(s)=X^{h_n}(s)-X^h(s),~~U^n(\overline{s}_m)=X^{h_n}(\overline{s}_m)-X^h(\overline{s}_m),
$$
where
\begin{eqnarray*}\label{eq:23}
\overline{s}_m=t_{k+1}\equiv (k+1)T\cdot 2^{-m}~~\text{for}~~s\in [kT2^{-m},(k+1)T2^{-m}[,
\end{eqnarray*}
and
\begin{eqnarray}\label{eqnarray7}
\sup_{t\in [0,T]}|I_n(t)|\leq \sum_{i=1}^4 \widetilde{I}_i+\widetilde{I}_5,
\end{eqnarray}
where
\begin{eqnarray*}\label{eq:18}
\widetilde{I}_1=\sup_{0\leq t\leq T}\Big|\int_0^t\big\langle B(s,X^h(s))\circ(h_n(s)-h(s)),U^n(s)-U^n(\overline{s}_m)\big\rangle_{F^*_{1,2}}ds\Big|,
\end{eqnarray*}

\begin{eqnarray*}\label{eq:19}
\widetilde{I}_2=\sup_{0\leq t\leq T}\Big|\int_0^t\Big\langle\big(B(s,X^h(s))-B(\overline{s}_m,X^h(s))\big)\circ(h_n(s)-h(s)),U^n(\overline{s}_m)\Big\rangle_{F^*_{1,2}}ds\Big|,
\end{eqnarray*}

\begin{eqnarray*}\label{eq:20}
\widetilde{I}_3=\sup_{0\leq t\leq T}\Big|\int_0^t\Big\langle\big(B(\overline{s}_m,X^h(s))-B(\overline{s}_m,X^h(\overline{s}_m))\big)\circ(h_n(s)-h(s)),U^n(\overline{s}_m)\Big\rangle_{F^*_{1,2}}ds\Big|,
\end{eqnarray*}

\begin{eqnarray*}\label{eq:21}
\widetilde{I}_4=\sup_{1\leq k\leq 2^m}\sup_{t_{k-1}\leq t\leq t_k}\Big|\Big\langle(B(t_k,X^h(t_k))\times \int_{t_{k-1}}^t(h_n(s)-h(s))ds, U^n(t_k)\Big\rangle_{F^*_{1,2}}\Big|,
\end{eqnarray*}

\begin{eqnarray*}\label{eq:22}
\widetilde{I}_5=\sum_{k=1}^{2^m}\Big|\Big\langle B(t_k,X^h(t_k))\cdot \int_{t_{k-1}}^{t_k}(h_n(s)-h(s))ds, U^n(t_k)\Big\rangle_{F^*_{1,2}}\Big|.
\end{eqnarray*}

Now, we estimate $\widetilde{I}_i$, $i=1,2,...,5$. Since
\begin{eqnarray}\label{eq:41}
U^n(s)-U^n(\overline{s}_m)&&\!\!\!\!\!\!\!\!=X^{h_n}(s)-X^h(s)-\big(X^{h_n}(\overline{s}_m)-X^h(\overline{s}_m)\big)\nonumber\\
&&\!\!\!\!\!\!\!\!=X^{h_n}(s)-X^{h_n}(\overline{s}_m)-\big(X^h(s)-X^h(\overline{s}_m)\big),
\end{eqnarray}
by \eref{eq:41} and H\"{o}lder's inequality, we get
\begin{eqnarray}\label{eq:26}
\widetilde{I}_1\leq&&\!\!\!\!\!\!\!\! C\int_0^T(\|X^h(s)\|_{F^*_{1,2}}+1)\cdot|h_n(s)-h(s)|_2\cdot\|U^n(s)-U^n(\overline{s}_m)\|_{F^*_{1,2}}ds\nonumber\\
\leq &&\!\!\!\!\!\!\!\!C\sup_{s\in [0,T]}(\|X^h(s)\|_{F^*_{1,2}}+1)\cdot \Big[\int_0^T\|X^{h_n}(s)-X^{h_n}(\overline{s}_m)\|_{F^*_{1,2}}\cdot |h_n(s)-h(s)|_2\nonumber\\
&&\!\!\!\!\!\!\!\!~~~~~~~~~~~~~~~~~~~~~~~~~~~~~~~~~~~~~~~~~+\|X^h(s)-X^h(\overline{s}_m)\|_{F^*_{1,2}}\cdot |h_n(s)-h(s)|_2ds\Big]\nonumber\\
\leq&&\!\!\!\!\!\!\!\!C\sup_{s\in[0,T]}(\|X^h(s)\|_{F^*_{1,2}}+1)\cdot \Big(\int_0^T|h_n(s)-h(s)|_2^2ds\Big)^{\frac{1}{2}}\cdot\nonumber\\
&&\!\!\!\!\!\!\!\!\Bigg\{\Big[\int_0^T\|X^{h_n}(s)-X^{h_n}(\overline{s}_m)\|^2_{F^*_{1,2}}ds\Big]^{\frac{1}{2}}+\Big[\int_0^T\|X^h(s)-X^h(\overline{s}_m)\|^2_{F^*_{1,2}}ds\Big]^{\frac{1}{2}}\Bigg\}.
\end{eqnarray}

To estimate $\int_0^T\|X^{h_n}(s)-X^{h_n}(\overline{s}_m)\|^2_{F^*_{1,2}}ds$, notice that from \eref{eq:4.1} we know
\begin{eqnarray}\label{eq:24}
X^{h_n}(\overline{s}_m)-X^{h_n}(s)-\int_{s}^{\overline{s}_m}L\Psi(X^{h_n}(t))dt=\int_s^{\overline{s}_m} B(t,X^{h_n}(t))\circ h_n(t)dt.
\end{eqnarray}
Applying the chain rule to $\|X^{h_n}(\overline{s}_m)-X^{h_n}(s)\|^2_{F^*_{1,2}}$, we get
\begin{eqnarray}\label{eq:25}
&&\!\!\!\!\!\!\!\!\|X^{h_n}(\overline{s}_m)-X^{h_n}(s)\|^2_{F^*_{1,2}}\nonumber\\
&&\!\!\!\!\!\!\!\!+2\int_s^{\overline{s}_m}~_{(L^2(\mu))^*}\big\langle(1-L)(\Psi(X^{h_n}(t))),X^{h_n}(t)-X^{h_n}(s)\big\rangle_{L^2(\mu)}dt\nonumber\\
=&&\!\!\!\!\!\!\!\!2\int_s^{\overline{s}_m}\big\langle\Psi(X^{h_n}(t)),X^{h_n}(t)-X^{h_n}(s)\big\rangle_{F^*_{1,2}} dt\nonumber\\
&&\!\!\!\!\!\!\!\!+2\int_s^{\overline{s}_m}\big\langle B(t,X^{h_n}(t))\circ h_n(t),X^{h_n}(t)-X^{h_n}(s)\big\rangle_{F^*_{1,2}} dt.
\end{eqnarray}
Integrating \eref{eq:25} over $[0,T]$ with respect to $s$, we obtain
\begin{eqnarray*}\label{eq:27}
&&\!\!\!\!\!\!\!\!\int_0^T\|X^{h_n}(\overline{s}_m)-X^{h_n}(s)\|^2_{F^*_{1,2}}ds\nonumber\\
&&\!\!\!\!\!\!\!\!+2\int_0^T\int_s^{\overline{s}_m}~_{(L^2(\mu))^*}\big\langle(1-L)(\Psi(X^{h_n}(t))),X^{h_n}(t)-X^{h_n}(s)\big\rangle_{L^2(\mu)} dtds\nonumber\\
=&&\!\!\!\!\!\!\!\!2\int_0^T\int_s^{\overline{s}_m}\big\langle\Psi(X^{h_n}(t)),X^{h_n}(t)-X^{h_n}(s)\big\rangle_{F^*_{1,2}} dtds\nonumber\\
&&\!\!\!\!\!\!\!\!+2\int_0^T\int_s^{\overline{s}_m}\big\langle B(t,X^{h_n}(t))\circ h_n(t),X^{h_n}(t)-X^{h_n}(s)\big\rangle_{F^*_{1,2}} dtds.
\end{eqnarray*}
Rewrite the above the equation as
\begin{eqnarray}\label{eq:28}
&&\!\!\!\!\!\!\!\!\int_0^T\|X^{h_n}(\overline{s}_m)-X^{h_n}(s)\|^2_{F^*_{1,2}}ds\nonumber\\
&&\!\!\!\!\!\!\!\!+2\int_0^T\int_s^{\overline{s}_m}~_{(L^2(\mu))^*}\big\langle(1-L)(\Psi(X^{h_n}(t))-\Psi(X^{h_n}(s))),X^{h_n}(t)-X^{h_n}(s)\big\rangle_{L^2(\mu)}dtds\nonumber\\
&&\!\!\!\!\!\!\!\!+2\int_0^T\int_s^{\overline{s}_m}~_{(L^2(\mu))^*}\big\langle(1-L)(\Psi(X^{h_n}(s))),X^{h_n}(t)-X^{h_n}(s)\big\rangle_{L^2(\mu)}dtds\nonumber\\
=&&\!\!\!\!\!\!\!\!2\int_0^T\int_s^{\overline{s}_m}\big\langle\Psi(X^{h_n}(t))-\Psi(X^{h_n}(s)),X^{h_n}(t)-X^{h_n}(s)\big\rangle_{F^*_{1,2}}dtds\nonumber\\
&&\!\!\!\!\!\!\!\!+2\int_0^T\int_s^{\overline{s}_m}\big\langle\Psi(X^{h_n}(s)),X^{h_n}(t)-X^{h_n}(s)\big\rangle_{F^*_{1,2}}dtds\nonumber\\
&&\!\!\!\!\!\!\!\!+2\int_0^T\int_s^{\overline{s}_m}\big\langle B(t,X^{h_n}(t))\circ h_n(t),X^{h_n}(t)-X^{h_n}(s)\big\rangle_{F^*_{1,2}} dtds.
\end{eqnarray}
By \eref{eqn7}, \eref{eqn1.1}, since $L^2(\mu)\subset F^*_{1,2}$ continuously and densely, by Young's inequality and \textbf{(H2)(ii)}, we obtain
\begin{eqnarray*}
&&\!\!\!\!\!\!\!\!\int_0^T\|X^{h_n}(\overline{s}_m)-X^{h_n}(s)\|^2_{F^*_{1,2}}ds\nonumber\\
\leq&&\!\!\!\!\!\!\!\!2\int_0^T\int_s^{\overline{s}_m}|\Psi(X^{h_n}(s))|_2\cdot|X^{h_n}(t)-X^{h_n}(s)|_2dtds+\frac{1}{\tilde{\alpha}}\int_0^T\int_s^{\overline{s}_m}\|X^{h_n}(t)-X^{h_n}(s)\|^2_{F^*_{1,2}}dtds\nonumber\\
&&\!\!\!\!\!\!\!\!+2\int_0^T\int_s^{\overline{s}_m}|\Psi(X^{h_n}(s))|_2\cdot\|X^{h_n}(t)-X^{h_n}(s)\|_{F^*_{1,2}}dtds\nonumber\\
&&\!\!\!\!\!\!\!\!+2C_2\int_0^T\int_s^{\overline{s}_m}(\|X^{h_n}(t)\|_{F^*_{1,2}}+1)\cdot |h_n(t)|_2\cdot\|X^{h_n}(t)-X^{h_n}(s)\|_{F^*_{1,2}}dtds,
\end{eqnarray*}
by \textbf{(H1)} we know that $|\Psi(X^{h_n}(s))|_2\leq C|X^{h_n}(s)|_2$, combining this with \eref{eqnarray3} and $\int_0^T|h_n(s)|_2^2ds\leq M$, the right hand-side of \eref{29} can be controlled by
\begin{eqnarray}\label{29}
\leq&&\!\!\!\!\!\!\!\!C_{T,M}\frac{1}{2^m}+C_{T,M}\int_0^T\int_s^{\overline{s}_m}|h_n(s)|_2dtds\nonumber\\
\leq&&\!\!\!\!\!\!\!\!C_{T,M}\frac{1}{2^m}+C_{T,M}\int_0^T\Big[\int_s^{\overline{s}_m}|h_n(s)|^2_2dt\Big]^{\frac{1}{2}}\cdot|\overline{s}_m-s|^{\frac{1}{2}}ds\nonumber\\
\leq&&\!\!\!\!\!\!\!\!C_{T,M}\sqrt{\frac{1}{2^m}},
\end{eqnarray}

similarly, we can get
\begin{eqnarray}\label{eq:42}
\int_0^T\|X^h(\overline{s}_m)-X^h(s)\|^2_{F^*_{1,2}}ds\leq C_{T,M}\sqrt{\frac{1}{2^m}},
\end{eqnarray}
so applying \eref{eqnarray3}, $\int_0^T|h_n(s)|_2^2ds\leq M$ and $\int_0^T|h(s)|_2^2ds\leq M$ again, \eref{eq:26}, \eref{29} and \eref{eq:42} imply
\begin{eqnarray}\label{eqnarray8}
\widetilde{I_1}\leq C_{T,M}2^{-\frac{m}{4}}.
\end{eqnarray}

By assumption \textbf{(H2)(ii)}, \textbf{(H3)} and \eref{eqnarray3},  we know
\begin{eqnarray}\label{31}
\widetilde{I_2}&&\!\!\!\!\!\!\!\!\leq\int_0^TC_2(1+\|X^h(s)\|_{F^*_{1,2}})\cdot |s-\overline{s}_m|^\gamma\cdot |h_n(s)-h(s)|_2\cdot\|U^n(\overline{s}_m)\|_{F^*_{1,2}}ds\nonumber\\
&&\!\!\!\!\!\!\!\!\leq C_{T,M}\cdot (\frac{T}{2^m})^\gamma\cdot\int_0^T|h_n(s)-h(s)|_2ds\nonumber\\
&&\!\!\!\!\!\!\!\!\leq C_{T,M} 2^{-m\gamma}.
\end{eqnarray}

By assumption \textbf{(H2)(i)}, \eref{eq:42} and \eref{eqnarray3} we have
\begin{eqnarray}\label{32}
\widetilde{I_3}\leq&&\!\!\!\!\!\!\!\!\int_0^TC_1\|X^h(s)-X^h(\overline{s}_m)\|_{F^*_{1,2}}\cdot|h_n(s)-h(s)|_2\cdot\|U^n(\overline{s}_m)\|_{F^*_{1,2}}ds\nonumber\\
\leq&&\!\!\!\!\!\!\!\!C_1\sup_{s\in[0,T]}\|U^n(s)\|_{F^*_{1,2}}\cdot\Bigg[\Big(\int_0^T\|X^h(s)-X^h(\overline{s}_m)\|^2_{F^*_{1,2}}ds\Big)^{\frac{1}{2}}\cdot\Big(\int_0^T|h_n(s)-h(s)|_2^2ds\Big)^{\frac{1}{2}}\Bigg]\nonumber\\
\leq&&\!\!\!\!\!\!\!\! C_{T,M}2^{-\frac{m}{4}}.
\end{eqnarray}

For $\widetilde{I_4}$, by assumption \textbf{(H2)(ii)} we know
\begin{eqnarray}\label{33}
\widetilde{I_4}\leq&&\!\!\!\!\!\!\!\! C_2\sup_{t\in [0,T]}\Big[\|U^n(s)\|_{F^*_{1,2}}\Big]\cdot\sup_{t\in[0,T]}\Big[\|X^h(s)\|_{F^*_{1,2}}+1\Big]\cdot\sup_{1\leq k\leq 2^m}\sup_{t_{k-1}\leq t\leq t_k}\int_{t_{k-1}}^t|h_n(s)-h(s)|_2ds\nonumber\\
\leq&&\!\!\!\!\!\!\!\!C_2\cdot \sqrt{\frac{T}{2^m}}\cdot \Big(\int_0^T|h_n(s)-h(s)|_2^2ds\Big)^{\frac{1}{2}}\nonumber\\
\leq&&\!\!\!\!\!\!\!\!C_{T,M}\cdot 2^{-\frac{m}{2}}.
\end{eqnarray}

Finally, notice that the weak convergence of $h_n$ to $h$ implies that for any $a,b\in[0,T]$, $a<b$, $\int_a^bh_n(s)ds$ converges to $\int_a^bh(s)ds$ in the weak topology of $L^2(\mu)$. Therefore, since the operator $B(t_k, X^h(t_k))$ is compact from $L^2(\mu)$ to $F^*_{1,2}$, we deduce that for every $k$,
$$\Big\|B(t_k,X^h(t_k))\times\Big(\int_{t_{k-1}}^{t_k}h_n(s)ds-\int_{t_{k-1}}^{t_k}h(s)ds\Big)\Big\|_{F^*_{1,2}}\rightarrow0,~~\text{as}~~n\rightarrow\infty.$$
Hence, for fixed $m$, $\widetilde{I_5}\rightarrow0$ as $n\rightarrow\infty$.

Now, taking \eref{eqnarray8}-\eref{33} into \eref{eqn1.2}, for arbitrary $m$, letting $n\rightarrow\infty$, we get
$$
\overline{\lim_{n\rightarrow\infty}}\sup_{t\in[0,T]}\|X^{h_n}(t)-X^h(t)\|^2_{F^*_{1,2}}\leq C_{T,M}(2^{-\frac{m}{2}}+2^{-\gamma m}+2^{-\frac{m}{4}}),
$$
so,
$$
\sup_{t\in[0,T]}\|X^{h_n}(t)-X^h(t)\|^2_{F^*_{1,2}}\rightarrow0,~~ \text{as}~~n\rightarrow\infty.
$$
Consequently, this completes the proof of Theorem 5.2.\hspace{\fill}$\Box$
\end{proof}

%Finally, from \cite{DWZZ}, we get the desired result of Theorem 3.1.

%%%%%%%%%%%%%%%%%%%%%%%%%%%%%%%%%%%%%%%%%%%%%%%%%%%%%%%%%%%%% the end of main theorem

\end{document}